\documentclass[11pt]{amsart} \textwidth=14.5cm \oddsidemargin=1cm
\usepackage{amsmath,wasysym}
\usepackage{amsxtra}
\usepackage{amscd}
\usepackage{amsthm}
\usepackage{amsfonts}
\usepackage{amssymb}
\usepackage{eucal}
\usepackage{graphics, color}
\usepackage{hyperref,mathrsfs}

\input prepictex
\input pictex
\input postpictex

\newtheorem{thm}{Theorem}[section]
\newtheorem{lem}[thm]{Lemma}
\newtheorem{cor}[thm]{Corollary}
\newtheorem{prop}[thm]{Proposition}

\newtheorem{proc}[thm]{Procedure}

\theoremstyle{definition}

\newtheorem{example}[thm]{Example}

\theoremstyle{remark}
\newtheorem{rem}[thm]{Remark}

\numberwithin{equation}{section}

\begin{document}

\newcommand{\thmref}[1]{Theorem~\ref{#1}}
\newcommand{\secref}[1]{Section~\ref{#1}}
\newcommand{\lemref}[1]{Lemma~\ref{#1}}
\newcommand{\propref}[1]{Proposition~\ref{#1}}
\newcommand{\corref}[1]{Corollary~\ref{#1}}
\newcommand{\remref}[1]{Remark~\ref{#1}}
\newcommand{\eqnref}[1]{(\ref{#1})}
\newcommand{\exref}[1]{Example~\ref{#1}}

\newcommand{\nc}{\newcommand}
\nc{\Z}{{\mathbb Z}}
\nc{\hZ}{{\underline{\mathbb Z}}}
\nc{\C}{{\mathbb C}}
\nc{\N}{{\mathbb N}}
\nc{\F}{{\mf F}}
\nc{\Q}{\mathbb {Q}}
\nc{\la}{\lambda}
\nc{\ep}{\delta}
\nc{\h}{\mathfrak h}
\nc{\n}{\mf n}
\nc{\A}{{\mf a}}
\nc{\G}{{\mathfrak g}}
\nc{\SG}{\overline{\mathfrak g}}
\nc{\DG}{\widetilde{\mathfrak g}}
\nc{\D}{\mc D}
\nc{\Li}{{\mc L}}
\nc{\La}{\Lambda}
\nc{\is}{{\mathbf i}}
\nc{\V}{\mf V}
\nc{\bi}{\bibitem}
\nc{\NS}{\mf N}
\nc{\dt}{\mathord{\hbox{${\frac{d}{d t}}$}}}
\nc{\E}{\EE}
\nc{\ba}{\tilde{\pa}}
\nc{\half}{\frac{1}{2}}
\nc{\mc}{\mathcal}
\nc{\mf}{\mathfrak} \nc{\hf}{\frac{1}{2}}
\nc{\hgl}{\widehat{\mathfrak{gl}}}
\nc{\gl}{{\mathfrak{gl}}}
\nc{\hz}{\hf+\Z}
\nc{\dinfty}{{\infty\vert\infty}}
\nc{\SLa}{\overline{\Lambda}}
\nc{\SF}{\overline{\mathfrak F}}
\nc{\SP}{\overline{\mathcal P}}
\nc{\U}{\mathfrak u}
\nc{\SU}{\overline{\mathfrak u}}
\nc{\ov}{\overline}
\nc{\wt}{\widetilde}
\nc{\wh}{\widehat}
\nc{\sL}{\ov{\mf{l}}}
\nc{\sP}{\ov{\mf{p}}}
\nc{\osp}{\mf{osp}}
\nc{\spo}{\mf{spo}}
\nc{\hosp}{\widehat{\mf{osp}}}
\nc{\hspo}{\widehat{\mf{spo}}}
\nc{\hh}{\widehat{\mf{h}}}
\nc{\even}{{\bar 0}}
\nc{\odd}{{\bar 1}}
\nc{\mscr}{\mathscr}

\newcommand{\blue}[1]{{\color{blue}#1}}
\newcommand{\red}[1]{{\color{red}#1}}
\newcommand{\green}[1]{{\color{green}#1}}
\newcommand{\white}[1]{{\color{white}#1}}


\newcommand{\aaf}{\mathfrak a}
\newcommand{\bb}{\mathfrak b}
\newcommand{\cc}{\mathfrak c}
\newcommand{\qq}{\mathfrak q}
\newcommand{\qn}{\mathfrak q (n)}
\newcommand{\UU}{\bold U}
\newcommand{\VV}{\mathbb V}
\newcommand{\WW}{\mathbb W}
\nc{\TT}{\mathbb T}
\nc{\EE}{\mathbb E}
\nc{\FF}{\mathbb F}
\nc{\KK}{\mathbb K}
\nc{\rl}{\texttt{r,l}}

 \advance\headheight by 2pt

\title[Quantum group of type $A$ and queer Lie superalgebra]
{Quantum group of type $A$ and representations of queer Lie superalgebra}

\author[Chen]{Chih-Whi Chen}
\address{Department of Mathematics, National Taiwan University, Taipei 10617, Taiwan}
\email{d00221002@ntu.edu.tw}

\author[Cheng]{Shun-Jen Cheng$^\dagger$}
\thanks{$^\dagger$Partially supported by a MOST and an Academia Sinica Investigator grant}
\address{Institute of Mathematics, Academia Sinica, Taipei,
Taiwan 10617} \email{chengsj@math.sinica.edu.tw}

\begin{abstract}
We establish a maximal parabolic version of the Kazhdan-Lusztig conjecture \cite[Conjecture 5.10]{CKW} for the BGG category $\mc{O}_{k,\zeta}$ of $\mf{q}(n)$-modules of ``$\pm \zeta$-weights'', where $k\leq n$ and $\zeta\in\C\setminus\hf\Z$. As a consequence, the irreducible characters of these $\mf q(n)$-modules in this maximal parabolic category are given by the Kazhdan-Lusztig polynomials of type $A$ Lie algebras. As an application, closed character formulas for a class of $\mf q(n)$-modules resembling polynomial and Kostant modules of the general linear Lie superalgebras are obtained.
\end{abstract}

\subjclass[2010]{17B67}

\maketitle

\section{Introduction}

Characters for certain classes of finite-dimensional irreducible modules over the queer Lie superalgebra $\mf{q}(n)$ were obtained in the classical works \cite{Pe, Sv2}.  The character problem of an arbitrary finite-dimensional irreducible $\mf{q}(n)$-module was then first solved by Penkov and Serganova \cite{PS1, PS2}. They provided an algorithm for computing the coefficient $a_{\la\mu}$ of the character of the irreducible $\mf{q}(n)$-module $L(\mu)$ in the expansion of the character of the associated Euler characteristic $E(\la)$ for given dominant weights $\la, \mu$.

In \cite{Br2} Brundan developed a rather different approach to computing the coefficient $a_{\la\mu}$ for integer dominant weights $\la,\mu$. Let $\mathbb{F}^n$ be the $\emph{n}$th exterior power of the natural representation of the type B quantum group of infinite rank (cf. \cite{JMO}). It was proved that the transition matrix $(a_{\la\mu})$, for $\la$ and $\mu$ dominant integer weights, is given by the transpose of the transition matrix between canonical and the natural monomial bases of $\mathbb{F}^n$ at $q=1$. This gives all irreducible characters of finite-dimensional integer weight modules in terms of a combinatorial algorithm for computing the canonical basis of $\mathbb F^n$. A new interpretation of the irreducible characters of finite-dimensional half-integer weight modules in the same spirit of Lusztig canonical basis of quantum groups was given in \cite{CK} and \cite{CKW} as well.

While finite-dimensional representations of the queer Lie superalgebra $\mf q(n)$ are now fairly well understood, their infinite-dimensional analogues have not been studied much in the literature. Except for $n=2$ and some special cases, e.g.~\cite{FM, Ch}, irreducible characters of infinite-dimensional modules in the BGG category remain largely unknown (see, e.g., the survey article \cite{GG}).

The Brundan-Kazhdan-Lusztig conjecture \cite[Conjecture 4.32]{Br1} for the BGG category of integer weight $\mf{gl}(m|n)$-modules was proved by Lam, Wang and the second author in \cite{CLW} (see also \cite{BLW}). In fact, in \cite{CLW, BLW} irreducible character problem for arbitrary Borel subalgebras was settled; see also \cite{CL} for algorithms.  Furthermore, in \cite{CMW}, by means of twisting functors and parabolic induction functors, Mazorchuk, Wang and the second author reduced the irreducible character problem for $\gl(m|n)$ of an arbitrary highest weight to that of an integer highest weight, for which the Brundan-Kazhdan-Lusztig conjecture is then applicable. This gives a complete solution of the irreducible character problem for the full BGG category.

A similar reduction is established for $\mf{q}(n)$ by the first author in \cite{Ch}. As a consequence, the problem of computing the characters of the irreducible modules of arbitrary weights in the BGG category $\mc{O}_{n}$ for $\mf{q}(n)$ is reduced to the irreducible character problem in the following three categories: (i) the BGG category $\mc{O}_{n,\mathbb{Z}}$ of $\mf{q}(n)$-modules of integer weights, see, e.g., \cite{Br2}. (ii)  the BGG category $\mc{O}_{n,\frac{1}{2}+\mathbb{Z}}$ of $\mf{q}(n)$-modules of half-integer weights, see, e.g., \cite{CK}, \cite{CKW}. (iii) ($\zeta\notin \mathbb{Z}/2$ and $k\in \{0,1,\ldots ,n\}$) the BGG category $\mc{O}_{n,{\zeta}^{k}}$ of $\mf{q}(n)$-modules of  "$\pm \zeta$-weights", see, \cite{CKW} or Section \ref{sec:par:cat}. {\em In the main body of the present paper, we shall use the notation $\mc{O}_{k,\zeta}$ to denote the category $\mc{O}_{n,{\zeta}^{k}}$, as $n$ will be fixed throughout}.

Kwon, Wang and the second author formulated a Kazhdan-Lusztig type conjecture for the BGG category in (iii) (\cite[Conjecture 5.10]{CKW}) above, analogous to Brundan's conjecture for the category (i) in \cite[Section 4.8]{Br2}. In the same paper, the authors also establish some connections between the canonical bases of types A,B,C. Their investigation seems to indicate connections between certain modules over $\mf{q}(n)$ and modules over the general linear Lie superalgebra $\mf{gl}(k|n-k)$ for various $k\leq n$.

In particular, for each $k\leq n$, one has a bijection between the highest weights of the irreducible objects in the category $\mc{O}_{n,{\zeta}^{k}}$ and those of the BGG category
of  integer-weight modules for $\mf{gl}(k|n-k)$, that is compatible with the linkage in both categories (see, e.g., \cite{Ch}). In fact, in \cite{Ch} it was proved that blocks of atypicality degree one of a certain maximal parabolic subcategory $\mc F_{k,\zeta}$ of $\mc{O}_{n,{\zeta}^{k}}$ are equivalent to blocks of atypicality degree one of the category of finite-dimensional modules over $\mf{gl}(k|n-k)$.

In the present paper, we study the Kazhdan-Lusztig conjecture for the BGG category $\mc{O}_{n,{\zeta}^{k}}$, formulated in \cite[Conjecture 5.10]{CKW}, which states that the irreducible characters for modules in $\mc{O}_{n,{\zeta}^{k}}$ are determined by the very same Brundan-Kazhdan-Lusztig polynomials as those for the BGG category of the general linear Lie superalgebra $\gl(k|n-k)$ of \cite{Br1}. The main result of the present paper is to (formulate and) prove a parabolic version of that conjecture for the maximal parabolic subcategory $\mc F_{k,\zeta}$ (see Section \ref{FormulationOfPKL}). We wish to point out that the irreducible $\mf q(n)$-modules in $\mc F_{k,\zeta}$ are almost always infinite-dimensional and the character formulas we have obtained in this paper are new.

The paper is organized as follows.  In Section \ref{SectionQuantumgroups}, we recall the quantum group of type A and the construction of the Fock space $\mc{E}^{m|n}$. We review the canonical and the dual canonical bases in (a topological completion of) $\mc{E}^{m|n}$, along with Brundan's algorithm for computing canonical basis.
Section \ref{SectionRepnOfLieSuperalg} is devoted to the study of representations of the queer Lie superalgebra $\mf{q}(n)$. Certain parabolic  subcategories of $\mc O_{k,\zeta}$ of $\mf{q}(n)$-modules are introduced and characterized.
In Section \ref{SectionTiltingModules} we study in detail the tilting modules in these parabolic subcategory $\mc O_{k,\zeta}$. The parabolic version of the Kazhdan-Lusztig conjecture for the maximal parabolic subcategory $\mc F_{k,\zeta}$ is then formulated precisely in Section \ref{FormulationOfPKL}. We establish a ``queer'' version of Serganova's fundamental lemma \cite[Theorem 5.5]{Ser} in Section \ref{SectionSerganovasFunLem}. This lemma is then used to prove the Kazhdan-Lusztig conjecture for $\mc F_{k,\zeta}$ in Section \ref{ProofOfMainThm}. Our proof here follows the idea of and is inspired by the proof of the main theorem in \cite{Br1}. Finally, we establish a closed Kac-Wakimoto type character formula for a class of $\mf q(n)$-modules in $\mc F_{k,\zeta}$ resembling ``Kostant modules'' for the general linear Lie superalgebra. For those $\mf{q}(n)$-modules resembling polynomial representations of the general linear Lie superalgebra we obtain a Sergeev-Pragacz type character formula as well. This is accomplished in Section \ref{sec:KW:formula}.

\subsection*{Acknowledgments} The results of the present paper were announced by the second author in the conference {\em Categorical Representation Theory and Combinatorics} held in KIAS in December 2015. In the same conference Brundan announced that he and Davidson can establish \cite[Conjecture 5.10]{CKW} in its full generality based on uniqueness of tensor product categorification in the same spirit as \cite{BLW}.

We have been informed by Shunsuke Tsuchioka that his computer calculations show that the conjectures for the irreducible characters of integer and half-integer weights in the full BGG category formulated respectively in \cite{Br2} and \cite{CKW} require corrections. We are indebted to him for kindly sharing his computations with us.

\subsection*{Notation} \label{SectionNotations} We use $\mathbb{N}$, $\mathbb{Z}$ and $\mathbb{Z}_{\geq 0}$ to denote the sets of natural numbers, integers, and non-negative integers, respectively. Here and below we let $m,n \in \mathbb{Z}_{\geq 0}$ and set \begin{align*} I(m|n):= \{ -m, -m+1, \ldots  ,-1\} \cup \{1, 2,\ldots ,n \}.\end{align*} Let $\mathbb{Z}^{m|n}$ be the set of all functions $f: I(m|n) \rightarrow \mathbb{Z}$.

For $p\in\N$, the symmetric group on $p$ letters is denoted by $\mf{S}_{p}$. Let $\mf{S}_{m|n}: = \mf{S}_{m} \times \mf{S}_n$. Note  that $\mf{S}_{m|n}$ acts on the right of $\mathbb{Z}^{m|n}$ by composition of functions.

Throughout the paper, we fix a complex number $\zeta \not\in\hf\Z$ which will be used from Section \ref{SubsectionZgradations} on.

\section{Quantum groups and combinatorial preliminaries} \label{SectionQuantumgroups}

In this section we recall the quantum group of type A of infinite rank. We refer to \cite[Section 2-c]{Br1} or \cite[Section 2]{CLW} for more details.

\subsection{Quantum group of type A}
Let $\UU := \UU_{q}(\mf{gl}_{\infty})$ be the quantum group of type A of infinite rank. This is the $\mathbb{Q}(q)$-algebra generated by $\{E_a, F_a, K_a, K_a^{-1}\}_{a\in \mathbb{Z}}$, subject to the relations \[K_aK_a^{-1} = K_a^{-1}K_a =1,\ \ K_aK_b = K_bK_a,\]
\[K_aE_bK_a^{-1} = q^{\delta_{a,b} - \delta_{a,b+1} }E_b, \ \  K_aF_bK_a^{-1} = q^{\delta_{a,b+1} - \delta_{a,b}  }F_b,\]
\[E_aF_b -F_bE_a = \delta_{a,b} \frac{K_{a,a+1}- K_{a+1,a}}{q-q^{-1}},\]
\[E_a E_b = E_bE_a, \hskip 106pt \text{ if $|a-b|>1$}, \]
\[E_a^2 E_b + E_bE_a^2 =(q+q^{-1})E_aE_bE_a, \hskip 10pt \text{ if $|a-b|=1$},\]
\[F_a F_b = F_bF_a, \hskip 111pt \text{ if $|a-b|>1$},\]
\[F_a^2 F_b + F_bF_a^2 =(q+q^{-1})F_aF_bF_a, \hskip 17pt \text{ if $|a-b|=1$}.\]
 Here and below $K_{a,b} := K_a K_b^{-1}$ for $a\neq b\in \mathbb{Z}$.

$\UU$ is a Hopf algebra with commultiplication $\Delta: \UU \rightarrow \UU \otimes \UU$  defined by
\[\Delta(E_a) = 1\otimes E_a + E_a \otimes K_{a+1,a},\]
\[\Delta(F_a) =  F_a \otimes 1 +K_{a,a+1}\otimes F_a,\]
\[\Delta(K_a) = K_a\otimes K_a,\]
for $a\in \mathbb{Z}$.

\subsection{Fock space $\mc{E}^{m|n}$} \label{SectionTheFockSpace}
Let $\VV$ be the natural $\UU$-module with basis $\{v_{a}\}_{a\in \mathbb{Z}}$  and let $\WW$ be its restricted dual with basis $\{w_a\}_{a\in \mathbb{Z}}$ normalized by $\langle w_a,v_b\rangle = (-q)^{-a}\delta_{a,b}$, for $a,b \in\ \mathbb{Z}$. The actions of $\UU$ on $\VV$ and $\WW$ are respectively given by
\begin{align*}
K_a v_b =q^{\delta_{a,b}}v_b, \ \ E_a v_b =\delta_{a+1,b} v_a, \ \ F_a v_b =\delta_{a,b}v_{a+1},\\
K_a w_b =q^{-\delta_{a,b}}w_b, \ \ E_a w_b =\delta_{a,b}w_{a+1}, \ \ F_a w_b =\delta_{a+1,b} w_a.
\end{align*}

For $m,n\in \mathbb{Z}_{\geq 0}$, the tensor space $ \TT^{m|n}:= \VV^{\otimes m}\otimes \WW^{\otimes n}$ can be viewed as a module over $\UU$ via the comultiplication $\Delta$. For $f\in \mathbb{Z}^{m|n}$, we let \[M_f:= v_{f(-m)}\otimes  v_{f(-m+1)} \otimes \cdots \otimes v_{f(-1)} \otimes w_{f(1)} \otimes w_{f(2)}  \otimes  w_{f(n)} \in \TT^{m|n}.\] The set $\{M_f\}_{f\in \mathbb{Z}^{m|n}}$ is referred to as the {\em standard monomial basis} for $\TT^{m|n}$.

Let $\mf{S}_{m}$ be the symmetric group on the letters in $I(m|0)$ with the set of generators $\{s_i:=(i \ \ i+1)|-m \leq i\leq -2\} \subseteq \mf{S}_{m}$. Recall that the {\em Iwahori-Hecke algebra} $\mathcal{H}_{m}$ associated to $\mf{S}_{m}$ is the associative $\mathbb{Q}(q)$-algebra generated by $H_i$, $-m\le i\le -2$, subject to the relations
 \[(H_i -q^{-1})(H_i + q) =0,\]
\[H_iH_{i+1}H_i = H_{i+1}H_iH_{i+1},\] \[H_iH_j=H_jH_i, \text{ for } |i-j|>1. \]
Denote the longest element in $\mf{S}_{m}$ by $\omega_{0}^{(m)}$. For each $\sigma\in \mathfrak{S}_{m}$, we have the corresponding element $H_{\sigma}: = H_{i_1}H_{i_2} \cdots H_{i_r}$ for any reduced expression $\sigma = s_{i_1}s_{i_2}\cdots s_{i_r}$. Recall that there is a unique antilinear ($q\rightarrow q^{-1}$) automorphism $\bar{\,}: \mc{H}_{m} \rightarrow \mc{H}_{m}$ such that $\overline{H_{\sigma}} = H^{-1}_{\sigma^{-1}}$, for all $\sigma \in \mf{S}_m$ (see, e.g., \cite[Section 2-e]{Br1}).

We denote by $\preceq_{\mf{a_m}}$ the classical Bruhat ordering on the weight lattice $\mathbb{Z}^m$ of $\mf{gl}(m)$. For $i\in I(m|n)$, let $d_i \in \mathbb{Z}^{m|n}$ be the function $j\mapsto -\text{sgn}(i)\delta_{ij}$. Recall the {\em super Bruhat ordering} $\preceq$ on $\mathbb{Z}^{m|n}$ for Lie superalgebra $\mf{gl}(m|n)$ defined in \cite[Section 2-b]{Br1} as follows.

Let $P$ be the free abelian group with basis $\{\epsilon_a\}_{a\in \mathbb{Z}}$. Let $\leq$ denote the partial ordering of weights on $P$, i.e., $f\leq g$ if and only if $f-g \in \sum_i \mathbb{Z}_{\geq 0}(\epsilon_i - \epsilon_{i+1})$. Let $\text{wt}_r^{\epsilon}: \mathbb{Z}^{m|n} \rightarrow P$ be the $\epsilon$-weight function defined by
\begin{align}
\text{wt}_r^{\epsilon}(f) := \sum_{r\leq i}-\text{sgn}(i)\epsilon_{f(i)}, \text{  for } f\in \mathbb{Z}^{m|n},\ \ r\in I(m|n).
\end{align}
The  super Bruhat ordering $\preceq$ on $\mathbb{Z}^{m|n}$ is defined by $f\preceq g$, if $\text{wt}^\epsilon_r f \leq \text{wt}^\epsilon_r g$, for all $r\in I(m|n)$, and $\text{wt}^\epsilon_{-m}f =\text{wt}^\epsilon_{-m} g$ (\cite[Section 2-b]{Br1}).

For $f\in \mathbb{Z}^{m|n}$, the {\em degree of atypicality} of $f$ is denoted by $\sharp f$ (see, e.g., \cite[(2.3)]{Br1}).
We say $f$ is {\em typical} if $\sharp f =0$; otherwise $f$ is {\em atypical}. For $f,g \in \mathbb{Z}^{m|n}$, we have that $f \succeq g$ implies $\sharp f =\sharp g$.

Recall that $\TT^{m|0}= \VV^{\otimes m}$ admits a $\UU$-$\mathcal{H}_{m}$-bimodule structure \cite{Jim}. Namely, on $\TT^{m|0}$ the Hecke algebra $\mc H_m$ acts as follows:
\begin{align} \label{HeckeAlgebaAction}
M_f H_i = \left\{ \begin{array}{lll}  M_{fs_i}, \text{ if } f \prec_{\mf{a_m}} fs_i,\\
q^{-1}M_{f}, \text{ if } f = fs_i,\\
M_{fs_i} -(q-q^{-1})M_f, \text{ if } f \succ_{\mf{a_m}} fs_i,
\end{array} \right.
\end{align}
for all $-m\leq i\leq -2$ and $ f\in \Z^{m|0}$. Similarly, we can define a $\UU$-$\mathcal{H}_{n}$-bimodule structure on  $\TT^{0|n} = \WW^{\otimes n}$.

For $p\in \N$, let
\begin{align*}
   H_0(p) : = \sum_{\sigma\in \mf{S}_{p}} (-q)^{\ell(\sigma) - \ell(\omega_0^{(p)})} H_{\sigma}.
\end{align*}
Then $H_0(p)$ is a bar-invariant element in $\mc H_p$ (see, e.g., \cite[Lemma 3.2]{Br1}).

It is proved in \cite[Propositions 1.1 and 1.2]{KMS} that $\TT^{m|0} =\text{Ker}H_0(m)|_{\TT^{m|0}} \oplus \TT^{m|0}H_0(m)|_{\TT^{m|0}}$ 
and $\text{Ker}H_0(m)|_{\TT^{m|0}} = \sum_{i=-m}^{-2} \text{Ker}(H_i-q^{-1})|_{\TT^{m|0}}$. Similarly, $\TT^{0|n}$ and $\text{Ker}H_0(n)|_{\TT^{0|n}}$ have analogous decompositions.

As a conclusion, $\TT^{m|n}$ admits a $\UU$-$(\mathcal{H}_{m}\otimes \mathcal{H}_{n})$ bimodule structure
(see, e.g., \cite[Section 2-e]{Br1}) with  $\TT^{m|n}= \text{Ker}H_0 \oplus \TT^{m|n} H_0$ and $\text{Ker}H_0 = \sum_i \text{Ker}(H_i-q^{-1})$, where the summation is over $i\in I(m|n)\setminus\{-1,n\}$ and $H_0:= H_{0}(m) H_{0}(n) \in \mathcal{H}_{m}\otimes \mathcal{H}_{n}$. We define the {\em Fock space} $\mc{E}^{m|n}:= \TT^{m|n} H_0$.

We can identify $\mc{E}^{m|n}$ with the {\em $q$-wedge space} $\wedge^m\VV\otimes\wedge^{n}\WW$ (see, e.g., \cite[Section 2.4]{CLW}). Let
\begin{align*}
\mathbb{Z}^{m|n}_+:= \{f\in \mathbb{Z}^{m|n}| f(-m)>f(-m+1)> \cdots  > f(-1),\ \  f(1)<f(2)< \cdots  < f(n) \}.\end{align*}
From \eqref{HeckeAlgebaAction}, it follows that $\{M_fH_0\}_{f\in \mathbb{Z}_+^{m|n}}$ forms a $\mathbb{Q}(q)$-basis for $\mc{E}^{m|n}$ and the following bijection from $\mc{E}^{m|n}$ to $\wedge^m\VV\otimes\wedge^{n}\WW$
\begin{align*}
M_fH_0 \mapsto  v_{f(-m)} \wedge \ldots \wedge v_{f(-1)}\otimes w_{f(1)} \wedge w_{f(2)} \wedge \ldots \wedge w_{f(n)}, \text{ for }f\in \mathbb{Z}^{m|n},
\end{align*}
gives an isomorphism of $\UU$-modules. For each $f\in \mathbb{Z}_+^{m|n}$, we define $K_f : = M_fH_0 \in \mc{E}^{m|n}$. We call $\{K_f\}_{\mathbb{Z}_+^{m|n}}$ the {\em standard monomial basis} for $\mc{E}^{m|n}$.

\subsection{Canonical and dual canonical bases of $\mc{E}^{m|n}$} \label{SectionBases}
We let $\widehat{\TT}^{m|n}$ and $\widehat{\mc{E}}^{m|n}$ denote certain topological completions of $\TT^{m|n}$ and ${\mc{E}}^{m|n}$, respectively, see \cite[Section 2-d]{Br1} for precise definition.
According to \cite[Theorem 2.14]{Br1}  $\widehat{\TT}^{m|n}$ admits a continuous, anti-linear bar-involution $\bar{\,}: \widehat{\TT}^{m|n} \rightarrow \widehat{\TT}^{m|n}$ such that $\overline{XuH} =\overline{X}\overline{u}\overline{H}$, for all $X\in \UU$, $u\in \widehat{\TT}^{m|n}$, $H\in \mathcal{H}_{m}\otimes \mathcal{H}_{n}$, and furthermore $\overline{M_f} = M_f$, for $f\in \mathbb{Z}^{m|n}$ with $f(-m)\leq \cdots \leq f(-1)$, $f(1) \geq \cdots \geq  f(n)$, $f(i) \neq f(j)$, for all $-m\leq i < 0 < j \leq n$.

\begin{thm} \label{KL-Lemma} \emph{(} \cite[Theorem 3.6]{Br1} \emph{)} There exist unique bar-invariant topological bases $\{U_f\}_{f\in \mathbb{Z}_+^{m|n}}, \{L_f\}_{f\in \mathbb{Z}_+^{m|n}}$ for $\widehat{\mc{E}}^{m|n}$ such that
\begin{align*}
U_f = K_f +\sum_{g\prec f} u_{g,f}(q)K_g, \ \ L_f = K_f +\sum_{g\prec f} \ell_{g,f}(q)K_g,
\end{align*}
where summation is over $g \in \mathbb{Z}^{m|n}_+$, $u_{g,f}(q)\in q\mathbb{Z}[q]$, and $\ell_{g,f}(q)\in q^{-1}\mathbb{Z}[q^{-1}]$.
\end{thm}

The topological bases $\{U_f\}_{f\in \mathbb{Z}_+^{m|n}}$ and $\{L_f\}_{f\in \mathbb{Z}_+^{m|n}}$ are respectively referred to as {\em canonical basis} and {\em dual canonical basis}  of $\widehat{\mc{E}}^{m|n}$ (see, e.g., \cite[Section 3-b]{Br1}). The polynomials $u_{g,f}(q)$, $\ell_{g,f}(q)$ can be computed combinatorially \cite[Corollary 3.39]{Br1}.

\subsection{Procedure for canonical basis} \label{CombinatorialSetup} We conclude this section with a review of \cite[Procedure 3.20]{Br1} for constructing canonical basis elements $U_f$, which indeed lie in $\mc E^{m|n}$. This will be used for construction of certain tilting modules of $\mf{q}(n)$ in Section \ref{SectionTiltingModules}.

\begin{proc} \label{Br1Procedure} \emph{(}\cite[Procedure 3.20]{Br1}\emph{)} \emph{ Assume that $f\in \mathbb{Z}_+^{m|n}$ with $\sharp f >0$. Define $h\in \mathbb{Z}_+^{m|n}$, $a\in \mathbb{Z}$ and $\widehat{X}_a, \widehat{Y}_a\in \{E_a,F_a\}$, by the following instructions below starting with step (I).}\end{proc}

{\bf Step (I)} Let $-m\leq i\leq -1$ be the largest number such that $f(i) = f(j)$ for some $1\leq j \leq n$. Go to step (II).

{\bf Step (II)} If $i\neq -1$ and $f(i)-1 =f(i+1)$, replace $i$ by $i+1$ and repeat step (II). Otherwise, go to step (III).

{\bf Step (III)} If $f(i)-1 = f(j)$ for some $1\leq j\leq n$, go to step (II*). Otherwise, set $\widehat{X}_a := F_{f(i)- 1}$, $\widehat{Y}_a:=E_{f(i)-1}$ and $h:= f- d_i$.

{\bf Step (II*)} If $j\neq 1$ and $f(j) -1 =f(j-1)$, replace $j$ by $j-1$ and repeat step (II*). Otherwise, go to step (III*).

{\bf Step (III*)}  If $f(j)-1 = f(i)$ for some $-m\leq i\leq -1$, go to step (II). Otherwise, set $\widehat{X}_a := E_{f(j)-1}$, $\widehat{Y}_a:=F_{f(j)-1}$ and $h:= f+ d_j$.

\vskip 8pt

After finitely many steps, the procedure reduces $f$ to a typical $g\in \mathbb{Z}_+^{m|n}$, namely, $U_f=\widehat{X}_{b_1}\widehat{X}_{b_2}\cdots \widehat{X}_{b_{\ell}}U_g$ for some $b_1,b_2,\ldots ,b_{\ell} \in \mathbb{Z}$ depending on $g$. Since $U_g = K_g$, we have the following lemma (cf. \cite[Lemma 3.19, 3.21]{Br1}).

\begin{lem} \label{LemmaForProcedure} Let $f$, $h$, $a\in \mathbb{Z}$, $\widehat{X}_a$ and $\widehat{Y}_a$ be given as above.
If $\sharp f =\sharp h$ then $\widehat{X}_a U_h =U_f, \widehat{Y}_a\widehat{X}_aU_h = U_h$ and $\widehat{X}_aK_h = K_f$. If $\sharp f  =\sharp h +1$ then we have  $\widehat{X}_a U_h =U_f,\ \ \widehat{Y}_a\widehat{X}_aU_h = (q+q^{-1})U_h$ and $\widehat{X}_aK_h = K_f +qK_{f-d_i+d_j}$, for some $-m\leq i<0 < j\leq n$, $f(i) = f(j)$ such that $f- d_i+d_j \in \mathbb{Z}^{m|n}_+$.
\end{lem}

\vskip 12pt

\section{Representations of the Lie superalgebra $\mf{q}(n)$} \label{SectionRepnOfLieSuperalg}

\subsection{Queer Lie superalgebra} \label{SectionQueerLieSuperalgebra}
Let $\mathbb{C}^{m|n}$ be the complex superspace of dimension $(m|n)$. The {\em general linear Lie superalgebra} $\mathfrak{gl}(m|n)$ may be realized as $(m+n) \times (m+n)$ complex matrices:
\begin{align} \label{glrealization}
 \left( \begin{array}{cc} A & B\\
C & D\\
\end{array} \right),
\end{align}
where $A,B,C$ and $D$ are respectively $m\times m, m\times n, n\times m, n\times n$ matrices. Let $E_{a,b}$ be the elementary matrix in $\mathfrak{gl}(m|n)$ with $(a,b)$-entry $1$ and other entries 0, for $a,b \in I(m|n)$ and let $\h'=\h'_{m|n}$ be the standard Cartan subalgebra of $\gl(m|n)$ spanned by the basis elements $\{E_{aa}\}$ and dual basis elements $\{\delta'_a\}$, for $a\in I(m|n)$. Denote by $\Phi'^+$ the set of positive roots in the standard Borel subalgebra.

For $m=n$, the subspace
\begin{align} \label{qnrealization}
 \mf{g}:= \mathfrak{q}(n)=
\left\{ \left( \begin{array}{cc} A & B\\
B & A\\
\end{array} \right) \middle\vert\ A, B: \ \ n\times n \text{ matrices} \right\}
\end{align}
forms a subalgebra of $\mathfrak{gl}(n|n)$ called the {\em queer Lie superalgebra}.

The set $\{e_{ij}, \bar{e}_{ij}|1\leq i,j \leq n\}$ is a linear basis for $\mathfrak{g}$, where $e_{ij}= E_{-n-1+i,-n-1+j}+E_{i,j}$ and $\bar{e}_{ij}= E_{-n-1+i,j}+E_{i,-n-1+j}$. Note that the even subalgebra $\mathfrak{g}_{\bar{0}}$ is spanned by $\{e_{ij}|1\leq i,j\leq n\}$, which is isomorphic to the general linear Lie algebra $\mathfrak{gl}(n)$.


Let $\mathfrak{h} = \mathfrak{h}_{\bar{0}}\oplus \mathfrak{h}_{\bar{1}}$ be the standard Cartan subalgebra of $\mathfrak{g}$, with linear bases $\{h_i:= e_{ii}| 1\leq  i \leq n\}$ and $\{\bar{h}_{i}:= \bar{e}_{ii}|1\leq i \leq n\}$ of $\mathfrak{h}_{\bar{0}}$ and $ \mathfrak{h}_{\bar{1}}$, respectively. Let $\{\delta_i| 1\leq i\leq n\}$ be the basis of $\mathfrak{h}_{\bar{0}}^{*}$ dual to $\{h_i|1\leq i\leq n\}$. We define a symmetric bilinear form $( ,) $ on $\mathfrak{h}_{\bar{0}}^{*}$ by $( \delta_i,\delta_j)  = \delta_{ij}$, for $1\leq i,j\leq n$.

 We denote by $\Phi,\Phi_{\bar{0}},\Phi_{\bar{1}}$ the sets of all roots, even roots and odd roots of $\mf{g}$, respectively. Let $\Phi^+=\Phi^+_\even\sqcup\Phi^+_\odd$ be the set of positive roots in its standard Borel subalgebra $\mf b=\mf b_\even\oplus\mf b_\odd$, which consists of matrices of the form \eqref{qnrealization} with $A$ and $B$ upper triangular. Ignoring parity we have $ \Phi_\even=\Phi_\odd = \{\delta_i - \delta_j| 1\leq i,j \leq n\}$ and $\Phi^+ = \{\delta_i- \delta_j| 1\leq i<  j \leq n\}$. We denote by $\leq$ the usual partial order on the weights $\mf{h}_{\bar{0}}^*$ defined by using $\Phi^+$. The Weyl group $W$ of $\mathfrak{g}$ is defined to be the Weyl group of the reductive Lie algebra $\mathfrak{g}_{\bar{0}}$ and hence acts naturally on $\mathfrak{h}_{\bar{0}}^*$ by permutation. For a given root $\alpha = \delta_i - \delta_j \in \Phi$, let $\bar{\alpha} :=\delta_i + \delta_j $.

In the this paper, $\mf{g}$-modules are always supposed to have compatible $\mathbb{Z}_2$-gradations with $\mf{g}$-actions, and $\mf{g}$-homomorphisms are not necessarily even. For a $\mathfrak{g}$-module $M$ and $\mu \in \mathfrak{h}_{\bar{0}}^*$, let $M_{\mu}:=\{m\in M| h\cdot m = \mu(h)m, \text{ for } h\in \mathfrak{h}_{\bar{0}}\}$ denote its $\mu$-weight space. If $M$ has a weight space decomposition $M = \oplus_{\mu\in \mathfrak{h}_{\bar{0}}^*}M_{\mu}$, its character is given as usual by $\text{ch}M = \sum_{\mu \in \mathfrak{h}_{\bar{0}}^*}\text{dim}M_{\mu}e^{\mu} $, where $e$ is an indeterminate. In particular, we have the root space decomposition $\mathfrak{g} = \mf{h}\oplus(\oplus_{\alpha \in \Phi} \mathfrak{g}_{\alpha})$ with respect to the adjoint representation of $\mathfrak{g}$.

Let $\la = \sum_{i=1}^{n} \la_i \delta_i  \in \mathfrak{h}_{\bar{0}}^*$, and consider the symmetric bilinear form on $\mathfrak{h}_{\bar{1}}^*$ defined by $\langle\cdot,\cdot\rangle_{\la} : = \la([\cdot,\cdot] )$. Let $\ell(\la)$ be the number of $i$'s with $\la_i \neq 0$.  Let $1\leq i_1<i_2< \cdots < i_{\ell(\la)}\leq n$  such that $\la_{i_1}, \la_{i_2}, \ldots , \la_{i_{\ell(\la)}}$ are non-zero. Denote by $\lceil \cdot \rceil$ the ceiling function. Let $\mathfrak{h}'_{\bar{1}}$ be a maximal isotropic subspace  of $\mathfrak{h}_{\bar{1}}$ associated to $\langle\cdot,\cdot\rangle_{\la} $. Put $\mathfrak{h}' =  \mathfrak{h}_{\bar{0}} \oplus \mathfrak{h}'_{\bar{1}}$. Let $\mathbb{C}v_{\la}$ be the one-dimensional $\mathfrak{h}'$-module with $h\cdot v_{\la} = \la(h)v_{\la}$ and $h' \cdot v_{\la} =0 $ for $h\in \mathfrak{h}_{\bar{0}}$, $h' \in \mathfrak{h}'_{\bar{1}}$. Then $I_{\la} : = \text{Ind}_{\mathfrak{h}'}^{\mathfrak{h}}\mathbb{C}v_{\la}$ is an irreducible $\mathfrak{h}$-module of dimension $2^{\lceil \ell(\la)/2 \rceil}$ (see, e.g., \cite[Section 1.5.4]{CW}). We let $\Delta(\la): = \text{Ind}_{\mathfrak{b}}^{\mathfrak{g}}I_{\la}$ be the {\em Verma module}, where $I_{\la}$ is extended to a $\mathfrak{b}$-module in a trivial way, and define $L(\la)$ to be the unique irreducible quotient of $\Delta(\la)$. Note that $\Delta(\la)$ and $L(\la)$ are unique up to $\mf{g}$-isomorphisms.

  For a weight $\la \in \mf{h}_{\bar{0}}^*$, we let $\sharp \la$ to be the {\em atypicality degree} of $\la$ (see, e.g., \cite[Definition 2.49]{CW}).
 We say $\la$ is {\em typical} if $\sharp \la =0$; otherwise $\la$ is called {\em atypical}.

\subsection{$\La_{k,\zeta}$-weights, $\mathbb{Z}$-gradations, and categories $\mc{HC}_{k,\zeta}(\mf{l}^{\rl})$} \label{SubsectionZgradations} Let $k,n\in \mathbb{Z}_{\geq 0}$ with $k\le n$ and $\zeta \in \mathbb{C}\backslash \frac{\mathbb{Z}}{2}$. We let \begin{align*}
&\La_{k,\zeta}:=\{\la=\sum_{r=1}^n\la_i\ep_i\in\h_\even^*|\la_i\in \zeta+\Z\text{ and }\la_j\in-\zeta+\Z,1\le i\le k < j\le n\},
\end{align*} 

Let $s,t\in\N$ such that $k=r_1+r_2+\ldots+r_s$ and $n-k=l_1+l_2+\ldots+l_t$, where $r_i,l_j\in\N$. Let $\texttt{r}=(r_1,\ldots,r_s)$ and $\texttt{l}=(l_1,\ldots,l_t)$, and put $\texttt{r}^c=\sum_{i=1}^cr_i$, and $\texttt{l}^d=k+\sum_{i=1}^d l_i$, for $c=0,\ldots,s$ and $d=0,\ldots,t$. We define ($c\not=s$ and $d\not=t$)
\begin{align*}
&\La_{k,\zeta}^{\texttt{r,l}}:=\{\la\in\La|\la_i-\la_{i+1}\in\N,\text{ for }\texttt{r}^c< i< \texttt{r}^{c+1}\text{ or }\texttt{l}^d< i< \texttt{l}^{d+1}\}.
\end{align*}

In the case $\texttt{r}=\underbrace{(1,1,\ldots,1)}_k$ and $\underbrace{\texttt{l}=(1,1,\ldots,1)}_{n-k}$ we have $\La^{(1,\ldots,1),(1,\ldots,1)}=\La_{k,\zeta}$. In the  case $\texttt{r}=(k)$ and $\texttt{l}=(n-k)$ we shall use the notation $\La^{+}_{k,\zeta}$ for $\La^{(k),(n-k)}_{k,\zeta}$, i.e.,
\begin{align*}
&\La_{k,\zeta}^{+}:=\{\la\in\La_{k,\zeta}|\la_i-\la_{i+1}\in\N, \text{ for }0< i< k \text{ or }k< i<n\}.
\end{align*}

Associated to $\La_{k,\zeta}^\texttt{r,l}$ we have a Levi subalgebra
\begin{align*}
\mf l^\rl=\bigoplus_{i=1}^s \mf{q}(r_i) \oplus \bigoplus_{j=1}^t \mf{q}(l_i)\subseteq \mf{q}(n),
\end{align*}
with corresponding parabolic subalgebra $\mf p^\rl$, nilradical $\mf u^\rl$, and opposite nilradical $\mf u^{\rl,-}$. Denote the roots in $\mf u^\rl$ by $\Phi^+(\mf u^\rl)$. If $\Pi$ denotes the set of simple roots of even positive roots $\Phi_\even$, we let $\Pi^\rl\subseteq\Pi$ denote the subset such that the even and odd root spaces $\G_\alpha\subseteq\mf l^\rl$ if and only if $\alpha\in\Pi^\rl$. Associated to $\La^\rl_{k,\zeta}$ we thus have a $\Z$-gradation of $\G=\bigoplus_{j\in\Z}\G_j$ uniquely determined by
\begin{align}\label{Z-gradation}
\deg\h=0,\quad \deg\G_{\pm\alpha}=0,\quad\deg\G_{\pm\beta}=\pm1,\quad\text{for } \alpha\in\Pi^\rl,\beta\in\Pi\setminus\Pi^\rl.
\end{align}
Note that this grading is also given by the formula \begin{align} \label{GradingOperators} [D,X] = jX, \ \ \text{ for }X\in \mf{g}_j, j\in \mathbb{Z},  \end{align} where $D$ is grading operator $\sum_{c=0}^{s-1} (n-c)\sum_{p=\texttt{r}^c+1}^{\texttt{r}^{c+1}} e_{pp} + \sum_{d=0}^{t-1} (n-s-d)\sum_{q=\texttt{l}^d+1}^{\texttt{l}^{d+1}} e_{qq}\in \mf{h}_{\bar{0}}$. Of course we have $\G_0 = \mf l^\rl$.

Let $W^\rl$ denote the Weyl group of $\mf l^\rl$, so that we have $W^\rl\cong\mf S_{r_1}\times\cdots\times\mf S_{r_s}\times\mf S_{l_1}\times\cdots\times\mf S_{l_t}$.  Let  $w^{\rl}_0$ be the longest element in  $W^{\rl}$ so that, for  $\la\in\La^{\rl}_{k,\zeta}$, we have  $-w^{\rl}_0\la\in\La^{\rl}_{k,-\zeta}$. In the case $\texttt{r}=\underbrace{(1,1,\ldots,1)}_k$ and $\underbrace{\texttt{l}=(1,1,\ldots,1)}_{n-k}$ we shall write $w_0$ for $w_0^{\rl}$, while in the  case $\texttt{r}=(k)$ and $\texttt{l}=(n-k)$ we shall write $w^+_0$ for $w_0^\rl$.

For given Levi subalgebra $\mf{s}$ of $\mf{g}$ containing $\h$, denote by $\mc{HC}_{k,\zeta}(\mf{s})$ the category of $\mf{s}$-modules that are direct sums of finite-dimensional simple $\mf{s}_{\bar{0}}$-modules with highest weights in $\La^{\rl}_{k,\zeta}$.

Let $\mf{b}^{\rl}$ be the standard Borel subalgebra of $\mf{ l^\rl}$, namely, $\mf{b}^{\rl}$ is generated by $\mf{h} \oplus (\oplus_{\alpha\in \Pi^\rl} \mf{g}_{\alpha})$. For given $\la\in\La^\rl_{k,\zeta}$, denote by  $\Delta^0(\la):= \text{Ind}_{\mf{b}^{\rl}}^{\mf l^\rl} I_{\la}$ the $\mf l^\rl$-Verma module of highest weight $\la$. Let $L^0(\la)$ be its unique irreducible quotient with highest weight $\la$. Note that $L^0(\la)$ is a typical $\mf l^\rl$-module and is furthermore finite dimensional.

\begin{lem} \label{CharacterizationOfLocFinModules}
$\mc{HC}_{k,\zeta}({\mf  l^{\emph{\rl}}})$ is a semisimple category with irreducible objects $\{L^0(\la)|\la\in\La^{\emph{\rl}}_{k,\zeta}\}$.
\end{lem}

\begin{proof}
It is enough to show that the full subcategory of $\mc{HC}_{k,\zeta}({\mf  l^{{\rl}}})$ consisting of objects with composition factors lying in $\{L^0(\la)|\la\in\La^{{\rl}}_{k,\zeta}\}$ is a semisimple category.

Observe that $L^0(\la)$ and $L^0(\mu)$ have different central characters for $\la, \mu\in\La^{{\rl}}_{k,\zeta}$ with $\la \neq \mu$ (see, e.g., \cite[Theorem 2.48]{CW}), and so there are no nontrivial extensions between these two irreducibles. Therefore, it suffices to show that $L^0(\la)$ has no self-extension in $\mc{HC}_{k,\zeta}({\mf  l^{{\rl}}})$, for every $\la\in\La^\rl_{k,\zeta}$. Suppose we have a short exact sequence of the form
\begin{align} \label{SESInLevi}
0\rightarrow L^0(\la) \rightarrow E \xrightarrow{f} L^0(\la) \rightarrow 0,
\end{align} in $\mc{HC}_{k,\zeta}({\mf  l^{{\rl}}})$.
Since $\mc{HC}_{k,\zeta}({\mf h})$ is a semisimple category (see, e.g., \cite[Lemma 1]{Fr}), \eqref{SESInLevi} implies that as $\h$-modules we have $E_\la=I_\la\oplus I_\la$. To distinguish these two copies let us write $E_\la=I_\la^{(1)}\oplus I^{(2)}_\la$, where we let $I^{(1)}_\la$ be highest weight space of the submodule $L^0(\la)$ in \eqref{SESInLevi}. Now consider the submodule $W=U(\mf l^\rl)I^{(2)}_\la\subseteq E$. Since $U(\mf l^\rl)I^{(1)}_\la=L^0(\la)$ is irreducible and $W_\la=I^{(2)}_\la$, we have $U(\mf l^\rl)I^{(2)}_\la\cap U(\mf l^\rl)I^{(1)}_\la=0$ and hence $E=W\oplus L^0(\la)$. It follows that $W\cong L^0(\la)$, and so \eqref{SESInLevi} is split.
\end{proof}

\subsection{Characters of irreducible $\mf l^\rl$-modules of $\La^\rl_{k,\zeta}$-highest weights} \label{Subsection::IrrChOfL0}
For $1\le c\le s$ and $1\le d\le t$, let
\begin{align*}
^{(c)}\la=(\la_{\texttt{r}^{c-1}+1},\ldots,\la_{\texttt{r}^{c}}) \quad \text{and}\quad \la^{(d)}=(\la_{\texttt{l}^{d-1}+1},\ldots,\la_{\texttt{l}^{d}}),
\end{align*}
regarded as weights in the even parts of the Cartan subalgebras of the corresponding queer Lie superalgebras $\mf{q}(r_c)$ and $\mf{q}(l_d)$, respectively. Then we have by Penkov's finite-dimensional typical character formula \cite[Theorem 2]{Pe}
\begin{align*}
\text{ch}L(\mf{q}(r_c),{^{(c)}\la})=2^{\lceil r_c/2\rceil}\prod_{\texttt{r}^{c-1}+1\le i<j\le \texttt{r}^c}\frac{({1+e^{{-\delta_i+\delta_j}}})}{({1-e^{{-\delta_i+\delta_j}}})}\sum_{w\in\mf S_{r_c}}(-1)^{\ell(w)}w(e^{{^{(c)}\la}}),\\
\text{ch}L(\mf{q}(l_d),\la^{(d)})=2^{\lceil l_d/2\rceil}\prod_{\texttt{l}^{d-1}+1\le s<t\le \texttt{l}^d}\frac{({1+e^{{-\delta_s+\delta_t}}})}{({1-e^{{-\delta_s+\delta_t}}})}\sum_{\sigma\in\mf S_{l_d}}(-1)^{\ell(\sigma)}\sigma(e^{\la^{(d)}}).
\end{align*} Therefore we obtain the following character formulas.
\begin{prop} \emph{(cf}. \cite[Section 3.1.3]{CW}\emph{)} \label{IrrChForLevi}
\begin{align*} \label{ChOfL0}
\emph{ch}L^0(\la)=&2^{\lceil n/2\rceil} \prod_{c=1}^s\prod_{\texttt{\emph{r}}^{c-1}+1\le i<j\le \texttt{\emph{r}}^c}\frac{({1+e^{{-\delta_i+\delta_j}}})}{({1-e^{{-\delta_i+\delta_j}}})}\sum_{w\in\mf S_{r_c}}(-1)^{\ell(w)}w(e^{{^{(c)}\la}})\\ &\prod_{d=1}^t\prod_{\texttt{\emph{l}}^{d-1}+1\le s<t\le \texttt{\emph{l}}^d}\frac{({1+e^{{-\delta_s+\delta_t}}})}{({1-e^{{-\delta_s+\delta_t}}})}\sum_{\sigma\in\mf S_{l_d}}(-1)^{\ell(\sigma)}\sigma(e^{\la^{(d)}}).
\end{align*}
\end{prop}

Recall the Levi subalgebra $\mf l^\rl$ with corresponding parabolic subalgebra $\mf{p}^\rl$, nilradicals $\mf u^\rl$, and opposite nilradical $\mf u^{\rl,-}$.
Observe that as an $\mf l^\rl$-module, we have
\begin{align*}
\mf u^{\rl,-}\cong& \bigoplus_{1\le i<j\le s}\hf\left[\C^{r_i|r_i*}\otimes\C^{r_j|r_j}\right]\oplus \bigoplus_{i,j}\hf\left[\C^{r_i|r_i*}\otimes\C^{l_j|l_j}\right]\oplus\\
& \bigoplus_{1\le i<j\le t}\hf\left[\C^{l_i|l_i*}\otimes\C^{l_j|l_j}\right].
\end{align*}
Above the factor $\hf$ is explained as follows: For given $p,q \in \mathbb{N}$,  both $\C^{p|p*}$ and $\C^{q|q}$ are so-called type $\texttt{Q}$ supermodules, and it is known that their tensor product is isomorphic to a direct sum of two copies of the same irreducible $\mf q(p)\oplus\mf{q}(q)$-module. The factor $\hf$ means that we take one copy of it, see, e.g., \cite[Section 3.1.3]{CW}.

\subsection{Parabolic BGG categories}\label{sec:par:cat}
Let $\mc{O}_n$ denote the BGG category of finitely generated $\mf{q}(n)$-modules which are locally finite over $\mf{b}$ and semisimple over $\mf{h}_{\bar{0}}$. In $\mc{O}_n$, we allow arbitrary  (not necessarily even) $\mf{g}$-morphisms. It is well-known that $\{L(\la)|\la \in \mf{h}^*_{\bar{0}}\}$ is a complete set of irreducible objects in $\mc{O}_n$, up to isomorphism. Let $\mc O^\rl_{k,\zeta}$ denote the full subcategory of $\mc{O}_n$ consisting of objects whose composition factors lie in $\{L(\la)| \la\in\La^\rl_{k,\zeta}\}$.  We shall use the following notations for the two extreme cases:
\begin{align*}
\mc O_{k,\zeta}:=\mc O^{(1,\ldots,1),(1,\ldots,1)}_{k,\zeta}, \quad \mc F_{k,\zeta}:=\mc O^{(k),(n-k)}_{k,\zeta}.
\end{align*}

Recall that $L^0(\la)$ denotes the finite-dimensional irreducible $\mf{l}^\rl$-module of highest weight $\la$ in Section \ref{Subsection::IrrChOfL0}. Note $L^0(\la)$ can be extended to a $\mf p^\rl$-module by letting $\mf u^\rl$ act trivially. Denote the corresponding {\em parabolic Verma module} by
\begin{align*}
\Delta^\rl(\la)=\text{Ind}_{\mf p^\rl}^\G L^0(\la).
\end{align*}

The following proposition is a characterization of the category $\mc O_{k,\zeta}^\rl$.

\begin{prop}
{\em $\mc O_{k,\zeta}^\rl$} is the full subcategory of $\mc O_n$ of {\em $\mf p^{\rl}$}-locally finite, completely reducible {\em $\mf l^\rl$}-modules of $\La^{\emph{\rl}}_{k,\zeta}$-highest weights.
\end{prop}

\begin{proof}
Let $\la\in\La^\rl_{k,\zeta}$. Note that $\Delta^\rl(\la)\cong \mc S\left(\mf u^{\rl,-}\right)\otimes L^0(\la)$ as an $\mf l^\rl$-module, where $\mc S\left(\mf u^{\rl,-}\right)$ denotes the supersymmetric tensor of $\mf u^{\rl,-}$. Since all the weights in $\mc S\left(\mf u^{\rl,-}\right)$ are integer weight, we see that all the $\mf l^\rl$-weights of $\Delta^\rl(\la)$ are $\mf l^\rl$-typical, and so $\Delta^\rl(\la)$ is a completely reducible $\mf l^\rl$-module by Lemma \ref{CharacterizationOfLocFinModules}. Therefore $\Delta^\rl(\la)$ is $\mf p^\rl$-locally finite and completely reducible over $\mf l^\rl$. Since $L(\la)$ is a quotient of $\Delta^\rl(\la)$, it follows that $L(\la)$ is also  $\mf p^\rl$-locally finite and completely reducible as a $\mf l^\rl$-module. This completes the proof.
\end{proof}

In the case $\texttt{r}=\underbrace{(1,1,\ldots,1)}_k$ and $\underbrace{\texttt{l}=(1,1,\ldots,1)}_{n-k}$ we shall write $\Delta(\la)$ for $\Delta^\rl(\la)$, which is consistent with earlier notation, while in the  case $\texttt{r}=(k)$ and $\texttt{l}=(n-k)$ we shall write $K(\la)$ for $\Delta^\rl(\la)$.

\begin{rem}
The $\mf q(n)$-module $L(\la)$, for $\la\in\La^\rl_{k,\zeta}$, is almost always infinite dimensional. Indeed, it follows from \cite[Theorem 4]{Pe} (see also \cite[Theorem 2.18]{CW}) that $L(\la)$ is finite dimensional if and only if $\la\in\La^+_{k,\zeta}$ and $k\in\{0,n\}$.
\end{rem}

\begin{rem} \label{ParabolicSubcategoryRem}
Basic features of parabolic subcategory for semisimple Lie algebras are well-known, see e.g., \cite[Chapter 9]{Hum08}. In the case of Lie superalgebras,  we refer to \cite{Mar14} in which the parabolic subcategory $\widetilde{\mc{O}}^{\mf p^{\rl}}$ corresponding to $\mf p^{\rl}$ is defined to be the full subcategory of $\mc{O}_{n,\bar{0}}$ consisting of $\mf{p}^{\rl}$-locally finite, and $\mf{l}^{\rl}_{\bar{0}}$-semisimple $\mf{q}(n)$-modules, where $\mc{O}_{n,\bar{0}}$ is the underlying even category of $\mc{O}_{n}$. Note that the underlying even category of $\mc O_{k,\zeta}^\rl$ is precisely the full subcategory of $\widetilde{\mc{O}}^{\mf p^{\rl}}$ consisting of $\mf{q}(n)$-modules of $\La_{k,\zeta}$-weights since each weight in $\La_{k,\zeta}$ is $\mf{l}^{\rl}$-typical. \end{rem}

\section{Tilting modules in parabolic categories} \label{SectionTiltingModules}

Let $k,n\in\Z_{\ge 0}$ with $k\leq n$ and $\zeta \in \mathbb{C}\backslash \hf\Z$ as before. In this section, we study tilting modules in $\mc F_{k,\zeta}$, and formulate the BGG reciprocity in terms of tilting modules by means of the Arkhipov-Soergel duality (see, e.g., \cite[Corollary 5.8]{Br3}).


For a given $\la \in \La^\rl_{k,\zeta}$, we recall the definition and existence of tilting modules $T^\rl(\la)$ in $\mc O^\rl_{k,\zeta}$, provided by \cite[Theorem 6.3]{Br3} (also see \cite[Section 4.3]{Mar14}).
In the case $\texttt{r}=\underbrace{(1,1,\ldots,1)}_k$ and $\underbrace{\texttt{l}=(1,1,\ldots,1)}_{n-k}$ (respectively, $\texttt{r}=(k)$ and $\texttt{l}=(n-k)$), i.e., $\la\in\La_{k,\zeta}$ (respectively, $\la\in\La^+_{k,\zeta}$), we denote the tilting module by
$T(\la)$ (respectively, $U(\la)$).

For given $m\in\N$, recall that $w_0^{(m)}$ denotes the longest element in $\mf S_m$.
The following lemma is well-known.

\begin{lem}\label{lem:dual:hwt}  Let $m\in \mathbb{N}$.
If $L(\la)$ be a finite-dimensional $\mf{q}(m)$-module then $L(\la)^*\cong L(-w_0^{(m)} \la )$.
\end{lem}
\begin{proof}
Since $L(\la)$ is finite-dimensional, $L(\la)$ is a direct sum of irreducible $\gl(m)$-modules with dominant highest weights $\mu$ such that $\la-\mu\in\sum_{\alpha\in\Phi^+}\Z_{\ge 0}\alpha$. Thus, the lowest $\gl(m)$-weight in $L(\la)$ is $w_0^{(m)}\la$, and hence $L(\la)^*$ has highest weight $-w_0^{(m)}\la$.
\end{proof}

Recall the supertrace $\text{str}_{V}(f)$ of an endomorphism $f= f_{\bar{0}}+ f_{\bar{1}}$ ($f_{\bar{0}}$ and $f_{\bar{1}}$ are respectively even and odd) on a superspace $V$ is defined by $\text{str}_{V}(f):= \text{tr}_{V_{\bar{0}}}f_{\bar{0}} -\text{tr}_{V_{\bar{1}}}f_{\bar{0}}$. We consider $\G=\bigoplus_{j\in\Z}\G_j$ with the $\Z$-gradation induced from \eqref{Z-gradation}. Recall that a Lie superalgebra homomorphism $\gamma: \G_{0} \rightarrow \mathbb{C}$ is called a semi-infinite character, if $\gamma([X,Y]) = \text{str}_{\G_{0}}(\text{ad}(X)\circ\text{ad}(Y))$, for $X\in \G_{1}, Y\in \G_{-1}$ (cf. \cite[Definition 1.1]{So} and \cite[Section 5]{Br3}). The proof of the following lemma is inspired by the proof of \cite[Lemma 7.4]{So}.

\begin{lem}\label{lem:0:semi}
The trivial character $0:\G_0\rightarrow \C$ is a semi-infinite character for the $\Z$-gradation \eqref{Z-gradation} for $\G$.
\end{lem}

\begin{proof}Let $X= X_{\bar{0}}+ X_{\bar{1}}$ and $Y= Y_{\bar{0}}+ Y_{\bar{1}}$  with $X_{\bar i}\in (\mf{g}_{1})_{\bar{i}},Y_{\bar i} \in (\mf{g}_{-1})_{\bar{i}}$ for $i =0,1$. We first note that $\text{str}_{\mf{g}_{0}}(\text{ad}X \circ \text{ad}Y) = \text{str}_{\mf{g}_{0}}(\text{ad}X_{\bar{0}} \circ \text{ad}Y_{\bar{0}}) + \text{str}_{\mf{g}_{0}}(\text{ad}X_{\bar{1}} \circ \text{ad}Y_{\bar{1}}) =\text{str}_{\mf{g}_{0}}(\text{ad}X_{\bar{1}} \circ \text{ad}Y_{\bar{1}})$, since $\mf{g}_{\bar{0}}$ and $\mf{g}_{\bar{1}}$ are isomorphic as $\mf{g}_{\bar{0}}$-modules. Thus, we may assume that $X\in (\G_1)_{\bar{1}}$, $Y \in (\G_{-1})_{\bar{1}}$.

Next, observe that, for each $A\in (\G_{0})_{\bar{0}}$, we have
 \[
 \text{str}_{\G_0}(\text{ad}[A,X]\circ \text{ad}Y) = \text{str}_{\G_0}(\text{ad}A\circ \text{ad}X\circ \text{ad}Y - \text{ad}X\circ \text{ad}A\circ \text{ad}Y) \]
 \[= \text{str}_{\G_0}(\text{ad}X\circ \text{ad}Y\circ \text{ad}A - \text{ad}X\circ \text{ad}A\circ \text{ad}Y)
 = \text{str}_{\G_0}(\text{ad}X\circ \text{ad}[Y,A]).\]
  Furthermore, since  $\G_{1}$ is a semisimple $\text{ad}(\G_{0})_{\bar{0}}$-module generated by root vectors of simple roots, it suffices to show that
 \begin{align} \label{EqStrIsZero}
 \text{str}_{\G_0}(\text{ad}X_{\alpha}\circ \text{ad}Y_{\beta}) =0,
 \end{align}
 for all $X_{\alpha} \in \G_{\alpha} \cap (\G_1)_\odd, Y_{\beta}\in \G_{\beta}\cap (\G_{-1})_\odd$ with $\alpha\in \Pi \setminus \Pi^{\rl}, \beta\in \Phi$.

Note that if $\alpha+\beta \neq 0$ then $(\text{ad}X_{\alpha}\circ \text{ad}Y_{\beta})(\G_{\gamma}) \subseteq  \G_{\alpha+\beta+\gamma} \neq \G_{\gamma}$ and so \eqref{EqStrIsZero} holds. Therefore we may assume that $\beta =- \alpha$.

Consider the triangular decomposition $\G_0=\mf{n}_{0}^+ \oplus \mf{h} \oplus \mf{n}_{0}^-$ of $\G_0$, with $\mf{n}_{0}^+:= \oplus_{\eta\in \Phi^+}(\G_0)_{\eta}$ and $\mf{n}_{0}^-:= \oplus_{\eta\in \Phi \setminus \Phi^+}(\G_0)_{\eta}$. Let $\Phi(\mf{n}_{0}^+)$ and $\Phi(\mf{n}_{0}^-)$ be the sets of roots of $\mf{n}_{0}^+$ and $\mf{n}_{0}^-$, respectively. Note that $\mf{n}_{0}^+$, $\mf{h}$ and $\mf{n}_{0}^-$ are stable under $\text{ad}X_{\alpha}\circ \text{ad}Y_{-\alpha}$. Furthermore,
 \begin{align*}
 \text{ad}X_{\alpha}(\mf{n}_{0}^-) \subset \G_{\alpha+\Phi(\mf{n}_{0}^-)}=0, \ \
 \text{ad}Y_{-\alpha}(\mf{n}_{0}^+) \subset \G_{-\alpha+\Phi(\mf{n}_{0}^+)}=0.
\end{align*}
Therefore we have
\begin{align*}
 \text{str}_{\G_0}(\text{ad}X_{\alpha}\circ \text{ad}Y_{-\alpha}) =  \text{str}_{\mf{h}}(\text{ad}X_{\alpha}\circ \text{ad}Y_{-\alpha}) +  \text{str}_{\mf{n}_{0}^-}(\text{ad}X_{\alpha}\circ \text{ad}Y_{-\alpha})  \\
 = \text{tr}_{\mf{h}_{\bar{0}}}(\text{ad}X_{\alpha}\circ \text{ad}Y_{-\alpha}) -  \text{tr}_{\mf{h}_{\bar{1}}}(\text{ad}X_{\alpha}\circ \text{ad}Y_{-\alpha}) +  \text{str}_{\mf{n}_{0}^-}(\text{ad}[X_{\alpha},Y_{-\alpha}]).
 \end{align*}
 Note that $[X_{\alpha},Y_{-\alpha}]\in \mf{h}_{\bar{0}}$ and so $\text{str}_{\mf{n}_{0}^-}(\text{ad}[X_{\alpha},Y_{-\alpha}]) =0$ since there is a natural isomorphisms between $(\mf{n}_{0}^-)_{\bar{0}}$ and $(\mf{n}_{0}^-)_{\bar{1}}$  as $\mf{h}_{\bar{0}}$-modules.

 Let $\pi: \mf{h}_{\bar{0}} \rightarrow \mf{h}_{\bar{1}}$ be the linear isomorphism defined by $\pi(e_{ii}) =\overline{e}_{ii}$, for $1\leq i \leq n$. Note that \begin{align*}\text{ad}X_{\alpha}\circ \text{ad}Y_{-\alpha}(h_{\bar{0}}) =  \alpha(h_{\bar{0}})  [X_{\alpha}, Y_{-\alpha}], \ \ \text{ad}X_{\alpha}\circ \text{ad}Y_{-\alpha}(h_{\bar{1}}) =  \overline{\alpha}(\pi(h_{\bar{1}}))  [X_{\alpha}, Y_{-\alpha}],\end{align*} for $i\in \{\bar{0},\bar{1}\}$ and $h_{i} \in \mf{h}_{i}$. It follows that $\text{tr}_{\mf{h}_{\bar{0}}}(\text{ad}X_{\alpha}\circ \text{ad}Y_{-\alpha})= \text{tr}_{\mf{h}_{\bar{1}}}(\text{ad}X_{\alpha}\circ \text{ad}Y_{-\alpha}) = 0$. This completes the proof.
\end{proof}

 Lemma \ref{lem:0:semi}, together with \cite[Theorem 6.4]{Br3} (c.f.~\cite[Theorem 5.12]{So}) and Lemma \ref{lem:dual:hwt}, implies the following tilting module version of the BGG reciprocity.
\begin{cor}\label{cor:tilting} For $\la ,\mu \in \La_{k,\zeta}^\rl$, we have
{\em \begin{align*}
\left(U(\la):K(\mu)\right) = [K(-w^\rl_0\mu):L(-w^\rl_0\la)].
\end{align*}}
\end{cor}

\section{Formulation of the Kazhdan-Lusztig conjecture in $\mc F_{k,\zeta}$}  \label{FormulationOfPKL}
Let $k,n\in\Z_{\ge 0}$ with $k\leq n$ and $\zeta \in \mathbb{C}\backslash \hf\Z$ as before. In \cite[Conjecture 5.10]{CKW} a Kazhdan-Lusztig type conjecture for $\mc{O}_{k,\zeta}$ was formulated in terms of canonical basis of $\TT^{m|n}$. In this section we formulate a parabolic version of the conjecture for $\mc{F}_{k, \zeta}$ in terms of canonical basis of $\mc E^{k|n-k}$.

We identify  $\La_{k,\zeta}$ with $\mathbb{Z}^{k|n-k}$ as follows:
For $\la \in \La_{k,\zeta}$,  we define $f_{\la} \in \mathbb{Z}^{k|n-k}$ by
  \begin{align}\label{aux:fns} f_{\la}(i) =
   \left\{ \begin{array}{ll}  \la_{i+k+1}- \zeta, \text{ if }  -k \leq i \leq -1, \\
  -(\la_{i+k}+ \zeta), \text{ if } 1 \leq i \leq n-k.
 \end{array} \right.
  \end{align}
This gives a bijection between $\La_{k,\zeta}$ and $\mathbb{Z}^{k|n-k}$, and furthermore under this bijection various definitions correspond, e.g., $\sharp f_{\la} = \sharp \la$. Also,  for a given $\mu \in \La_{k,\zeta}$, we let $\la \preceq \mu$ if $f_{\la} \preceq f_{\mu}$. Note that $\la \preceq \mu$ implies $\la \leq \mu$, for all $\la,\mu \in\La_{k,\zeta}$. Under this bijection the set $\La_{k,\zeta}^+$ is sent to $\mathbb{Z}^{k|n-k}_+$ so that we can identity these two sets.

Recall the canonical and dual canonical bases in Section \ref{SectionBases}.  For $\la, \mu \in \La^+_{k,\zeta}$, we define $\ell_{\la,\mu}(q): =\ell_{f_{\la}, g_{\mu}}(q)$ and $u_{\la,\mu}(q): =u_{f_{\la}, g_{\mu}}(q)$, where $\ell_{g,f}(q)$ and $u_{g,f}(q)$ are as in Theorem \ref{KL-Lemma}. We have the following parabolic version of \cite[Conjecture 5.10]{CKW} for $\mc F_{k,\zeta}$, whose proof will be given in Section \ref{ProofOfMainThm}.

\begin{thm} \label{PKLConjecture} For $\la \in \La_{k,\zeta}^+$, we have
\begin{align*} [U(\la)] = \sum_{\mu \preceq \la, \mu\in \La_{k,\zeta}^+} u_{\mu\la}(1)[K(\mu)],\\ [L(\la)] = \sum_{\mu \preceq \la,\mu\in \La_{k,\zeta}^+} \ell_{\mu\la}(1)[K(\mu)]. \end{align*}
\end{thm}

\section{Serganova's fundamental lemma for $\mc F_{k,\zeta}$} \label{SectionSerganovasFunLem}

Let $k,n\in\Z_{\ge 0}$ with $k\leq n$ and $\zeta \in \mathbb{C}\backslash \hf\Z$ as before. In this section we shall prove the queer Lie superalgebra version of Serganova's fundamental lemma \cite[Theorem 5.5]{Ser}. Such a ``queer'' version for the category $\mc F_{k,\zeta}$ is needed for the purpose of adapting Brundan's proof of his finite-dimensional irreducible character formula for the general linear Lie superalgebra \cite[Theorem 4.37]{Br1} to our setting of queer Lie superalgebra.

Recall that $\ov{\alpha}:=\ep_i+\ep_j$,  for a given $\alpha=\ep_i-\ep_j\in\Phi^+$ (Section \ref{SectionQueerLieSuperalgebra}). We first recall the following lemma of Penkov and Serganova:
\begin{lem}\label{lem:verma:hom} \cite[Proposition 2.1]{PS2} Let $\alpha\in\Phi^+$ and suppose that $(\la,\ov{\alpha})=0$. Then
$$\rm{Hom}_{\G}(\Delta(\la-\alpha),\Delta(\la))\not=0.$$
\end{lem}

The following theorem and its proof are inspired by \cite[Theorem 5.5]{Ser}.

\begin{thm}\label{thm:fund:fact}
Let  $\la\in\La^+_{k,\zeta}$. Suppose that $\alpha\in\Phi^+$ such that $(\la,\ov{\alpha})=0$ and $\la-\alpha\in\La^+_{k,\zeta}$. Then
\begin{align*}
\rm{Hom}_\G\left(K(\la-\alpha),K(\la)\right)\not=0.
\end{align*}
In particular, $[K(\la):L(\la-\alpha)]\not=0$.
\end{thm}

\begin{proof} In this proof we shall respectively denote  $\mf p^\rl$, $\mf l^\rl$ and $\mf u^\rl$ by $\mf p$, $\mf l$ and $\mf u$.

First we have an exact sequence of $\mf l$-modules
\begin{align*}
0\longrightarrow I^0(\la)\longrightarrow \Delta^0(\la)\longrightarrow L^0(\la)\longrightarrow 0,
\end{align*}
where $\Delta^0(\la)$ denotes the $\mf l$-Verma module of highest weight $\la$ (Section \ref{SubsectionZgradations}). This exact sequence trivially extends to an exact sequence of $\mf p$-module by letting $\mf u$ act trivially, and thus we have an exact sequence of $\G$-modules by parabolic induction
\begin{align*}
0\longrightarrow \text{Ind}^\G_{\mf p}I^0(\la)\longrightarrow \Delta(\la)\longrightarrow K(\la)\longrightarrow 0.
\end{align*}
By Lemma \ref{lem:verma:hom} we have
\begin{align*}
\text{Hom}_{\G}(\Delta(\la-\alpha),\Delta(\la))\not=0,
\end{align*}
and thus there exists a non-zero $\mf b$-singular vector $v_{\la-\alpha}\in\Delta(\la)$. It suffices to show that $v_{\la-\alpha}\not\in \text{Ind}^\G_{\mf p}I^0(\la)$.

Suppose on the contrary that $v_{\la-\alpha}\in \text{Ind}^\G_{\mf p}I^0(\la)$. Now $v_{\la-\alpha}$ is of course $\mf b_\even$-singular. We observe that if $\mu\in\h^*$ is the highest weight of a composition factor in $I^0(\la)$, then
\begin{align*}
\mu=w(\la),
\end{align*}
for some $w\in \mf S_k\times\mf S_{n-k}$. This is a direct consequence of \cite[Theorem 1]{FM}, according to which we have an equivalence of categories between strongly typical blocks of  $\mf{q}(k)\oplus\mf{q}(n-k)$-modules and the corresponding blocks of $\mf{gl}(k)\oplus\mf{gl}(n-k)$-modules.

Thus, any weight $\mu$ of a $\mf b_\even$-singular vector in $\text{Ind}^\G_{\mf p}I^0(\la)$ is of the form
\begin{align*}
\mu=w(\la)-\gamma,
\end{align*}
where $\gamma$ is a linear $\Z_{\geq 0}$-combination roots in $\Phi^+(\mf{u})$. Thus, we have
\begin{align*}
\mu=\la-\eta-\gamma,
\end{align*}
where $\eta$ is a $\Z_{\geq 0}$-linear combination of positive roots of $\mf l$. Thus, by assumption we have $\la-\alpha=\la-\eta-\gamma$ and so
\begin{align}\label{aux100}
\alpha=\eta+\gamma.
\end{align}
Now, $\alpha$ is a root in $\mf u$, and so \eqref{aux100} implies that $\gamma\in\Phi^+(\mf u)$, and there are three possibilities for $\eta$:
\begin{align*}
\eta=\begin{cases}\delta_i-\delta_s+\delta_t-\delta_j,\quad 1\le i<s\le k,k+1\le t<j\le n,\cr
\delta_i-\delta_s,\quad 1\le i<s\le k,\cr
\delta_t-\delta_j,\quad k+1\le t<j\le n.
\end{cases}
\end{align*}

Let us first consider the case $\eta=\delta_i-\delta_s$, with $1\le i<s\le k$. Thus, we have $w(\la)=\la-\delta_i+\delta_s$. Now we have $w\in\mf S_k\times\mf S_{n-k}$, and also all the $\la_i$s are distinct, for $1\le i\le k$. Thus, we must have
\begin{align*}
\la_i-1=\la_s.
\end{align*}
Therefore, we have $(\la,\eta)=\la_i-\la_s=1$ and $(\alpha,\eta)=1$, so that we have $(\la-\alpha,\delta_i-\delta_s)=0$. But then $\la-\alpha\not\in\La^+_{k,\zeta}$, which is a contradiction.

By a similar argument, the case $\eta=\delta_t-\delta_j$ with $k+1\le i<s\le n$ leads to a contradiction as well.

Finally, we assume that $\eta = \delta_i-\delta_s+\delta_t-\delta_j$, for some $1\le i<s\le k$ and $ k+1\le t<j\le n$. In this case, we have $\gamma = \delta_s - \delta_t$. Similarly, since each component of $\la$ are distinct, it follows from $w\in \mf S_k\times\mf S_{n-k}$ that $\la_ i -1 = \la_s$ and $\la_t -1 = \la_j$. Therefore, $ (\la,\eta)= \la_i- \la_s+\la_t-\la_j =2$ and $(\alpha,\eta)=2$. Now $(\la-\alpha,\delta_i-\delta_s) +(\la-\alpha,\delta_t-\delta_j) = (\la-\alpha,\eta)=0$, which also leads to $\la-\alpha\not\in\La^+_{k,\zeta}$.
\end{proof}

\section{Proof of the main theorem} \label{ProofOfMainThm}

Let $k,n\in\Z_{\ge 0}$ with $k\leq n$. Recall that $\zeta \in \mathbb{C}\backslash \hf\Z$ is fixed in Section \ref{SectionNotations}, and the free abelian group $P = \oplus_{a\in \mathbb{Z}}\mathbb{Z}\epsilon_a$ is defined in  Section \ref{SectionTheFockSpace}. We let $\mc{F}:=\mc{F}_{k,\zeta}$ in this section.

Let $\text{wt}: \Lambda^+_{k,\zeta} \rightarrow P$ be the weight function defined by (c.f. \cite[Section 2-c]{Br2})
\begin{align*}
\text{wt}(\la)  := \sum_{i=1}^{k}\epsilon_{\la_{i}-\zeta}  - \sum_{i=k+1}^{n}\epsilon_{-(\la_{i}+\zeta)}.
\end{align*}
It is well-known that $\chi_{\la} = \chi_{\mu}$ if and only if $\text{wt}(\la) = \text{wt}(\mu)$ (see, e.g., \cite[Theorem 2.48]{CW}). We have  decomposition $\mc F = \oplus_{\la\in \mf{h}_{\bar{0}}^*} \mc{F}_{\chi_{\la}} = \oplus_{\gamma \in P} \mc{F}_{\gamma}$  according to central characters $\chi_{\la}$ with $\text{wt}(\la) = \gamma$.

 Let $\mathbb{C}^{n|n}$ and $(\mathbb{C}^{n|n})^*$  be the standard representation and its dual, respectively. Denote the projection functor from $\mc{F}$ to $\mc{F}_{\gamma}$ by $\text{pr}_{\gamma}$.  We define the {\em translation functors} $\text{E}_{a}, \text{F}_{a}: \mc{F} \rightarrow \mc{F}$ as follows
\begin{align}
\text{E}_{a} (M):= \text{pr}_{\gamma+(\epsilon_a-\epsilon_{a+1})}(M\otimes (\mathbb{C}^{n|n})^*), \ \
\text{F}_{a} (M):= \text{pr}_{\gamma-(\epsilon_a-\epsilon_{a+1})}(M\otimes \mathbb{C}^{n|n}),
\end{align}
for all  $M\in  \mc{F}_{\gamma}$, $\gamma \in P$ , $a\in \mathbb{Z}$. For each $a\in \mathbb{Z}$ , it is not hard to see that both $\text{E}_{a}$ and $\text{F}_{a}$ are exact and bi-adjoint to each other.
We write $\la \rightarrow_{a} \mu$ if $\la , \mu \in \Lambda^+_{k,\zeta}$ and there exists $1\leq i \leq k$ such that $\la_i  = \mu_i-1 = a+\zeta$  or there exists $k+1 \leq i'\leq n$ such that $\la_{i'} = \mu_{i'}-1 = -a -1 -\zeta $, and in addition, $\la_j = \mu _j$ for all $ j\neq i$ in the former case, for all $j\neq i'$ in the later case. Let $\mc{K}(\mc{F})$ be the Grothedieck group of $\mc{F}$ and denote the element corresponding to $M\in \mc{F}$ by $[M]$.

We have the following lemma \cite[Lemma 4.2]{Ch}.

\begin{lem} \label{chOfKInTranslationFunctor}
Let $\la\in \Lambda^+_{k,\zeta}$. Then both $\emph{E}_{a} K(\la)$ and $\emph{F}_{a} K(\la)$ have flags of parabolic Verma modules and  we have the following formula:
\begin{align*}
[\emph{E}_{a} K(\la)] = 2\sum_{\mu \rightarrow_{a} \la} [K(\mu)], \quad
[\emph{F}_{a} K(\la)] = 2\sum_{\la \rightarrow_{a} \mu} [K(\mu)].
\end{align*}
\end{lem}

 We defined the $\mathbb{Z}$-form $\mc{E}_{\mathbb{Z}}^{k|n-k}$ of $\mc{E}^{k|n-k}$, namely, $\mc{E}_{\mathbb{Z}}^{k|n-k}: = \mathbb{Z} \otimes_{\mathbb{Z}[q,q^{-1}]}\mc{E}_{\mathbb{Z}[q,q^{-1}]}^{k|n-k}$ by letting $q=1$, where $\mc{E}_{\mathbb{Z}[q,q^{-1}]}^{k|n-k}$ is the $\mathbb{Z}[q,q^{-1}]$-lattice spanned by $\{K_f\}_{f\in \mathbb{Z}^{k|n-k}_+}$, and for given $f\in \La^+_{k,\zeta}$ we let $K_{f}(1):= 1\otimes K_{f}, U_f(1):=1\otimes U_f\in \mc{E}_{\mathbb{Z}}^{k|n-k}$.

Let  $\mc{A}_{k|n-k}^{\Delta}$ be the full subcategory of finite-dimensional modules over the general linear Lie superalgebra $\mf{gl}(k|n-k)$ consisting of objects that have a flag of Kac modules, see \cite[Sections 4-a,b]{Br1}. Recall that $\mc{A}_{k|n-k}^{\Delta}$ is also equipped with translation functors (see e.g., \cite[Section 4-b]{Br1} and \cite[Sections 3.4 and 5.1]{CW08}). Let $ \mc{F}^{\Delta} $ be the full subcategory of $\mc{F}$ of all modules which have a flag of $K(\la)$ with $\la \in \La^+_{k,\zeta}$. Let $\mc K(\mc F^\Delta)$ be the Grothendieck group of $\mc F^\Delta$. Now Lemma \ref{chOfKInTranslationFunctor}, together with \cite[Corollary 4.26 and Theorem 4.28]{Br1}, implies the following proposition that says that the translation functors for $\mc{F}^{\Delta}$ is the same as the translation functors on $\mc{A}_{k|n-k}^{\Delta}$ on the level of Grothendieck groups up to a $2$-factor.

\begin{prop} \label{PropIsoBetweenKandFockSpace}
Let $j:\mc K(\mc F^\Delta)\rightarrow \mc E^{k|n-k}_\Z$ be the $\Z$-isomorphism defined by
\begin{align} \label{IsoBetweenKandFockSpace}
j([\text{K}(\la)]) = \text{K}_{f_{\la}}(1) , \ \  \text{  for } \la\in \Lambda^+_{k,\zeta}.
\end{align} Then the representation theoretically defined functors $\emph{F}_a$ and $\emph{E}_{a}$ on $\mc{F}$ decategorify to the Chevalley generators $2F_a$ and $2E_a$ of $\UU_{q}(\mf{gl}_{\infty})|_{q=1}$ on $\mc{E}_{\mathbb{Z}}^{k|n-k}$.
\end{prop}

\begin{prop} \label{TypicalTilting}
Let $\la\in\La^+_{k,\zeta}$. If $\la$ is typical, then $K(\la)=L(\la)=U(\la)$.
\end{prop}
\begin{proof}
We have a surjection $K(\la)\rightarrow L(\la)$ that sends the highest weight space to the highest weight space. Now, if $K(\la)$ has a singular vector, then its weight $\mu$ lies $\La^+_{k,\zeta}$ and furthermore we have identical central character $\chi_\la=\chi_\mu$. Since $\la$ is typical, we must have $\la=\mu$. Thus, $K(\la)=L(\la)$ is irreducible.

Note that $\la\in\La^+_{k,\zeta}$ is typical if and only if $-w_0^+\la\in\La^+_{k,-\zeta}$ is typical. Thus, we have $K(-w_0^+\la)=L(-w_0^+\la)$, and hence by Corollary \ref{cor:tilting}, we have $U(\la)=K(\la)$.
\end{proof}

Let $\la\in \La^+_{k,\zeta}$ and $a\in \mathbb{Z}$. It is known that both $\text{E}_aU(\la)$ and $\text{F}_aU(\la)$ are direct sums of tilting modules (see, e.g., \cite[Corollary 4.27]{Br1}). Furthermore, we have the following lemma \cite[Lemma 4.3]{Ch}.

\begin{lem} \label{TransFunOfTiltings}
Let $\la\in \La^+_{k,\zeta}$. Then the multiplicity of each non-zero tilting module summand of $\emph{E}_aU(\la)$ and $\emph{F}_aU(\la)$ is even.
\end{lem}

The following lemma follows  from Procedure \ref{Br1Procedure}.

\begin{lem} \label{ExpressionOfUf} For every $f\in \mathbb{Z}^{k|n-k}_{+}$, we have 
$U_f(1) \in K_f(1) + \sum_{g \prec f}\mathbb{Z}_{\geq 0} K_g(1)$.
\end{lem}

We have now all the ingredients to adapt Method two of the proof of \cite[Theorem 4.37]{Br1} to prove that
 Procedure \ref{Br1Procedure}  specialized at $q=1$ gives the construction of the tilting modules in $\mc F$.
\begin{thm} \label{ConstructionOfTiltings}
Let $\la\in\La^+_{k,\zeta}$. Then $[U(\la)]$ is mapped to $U_{f_{\la}}(1)$ under the isomorphism $j$ in \eqref{IsoBetweenKandFockSpace}.
\end{thm}

\begin{proof}
We shall proceed by induction on the degree of atypcality $\sharp\la$ of $\la$. If $\sharp\la=0$, then $K(\la)=L(\la)=U(\la)$ by Lemma \ref{TypicalTilting}. Assume that $\sharp \la >0$ and furthermore $j([U(\nu)]) = U_{h}(1)$, where $\nu \in \La^+_{k,\zeta}$ satisfies $h = f_{\nu}$. Let $\widehat{X}_a \in \{E_a, F_a\}_{a\in \mathbb{Z}}$ be the operators given in Procedure \ref{Br1Procedure}. For each tilting module $U\in \mc{F}$ we define $X_aU$ to be a direct summand of the direct sum of two isomorphic copies of $\widehat{X}_aU$ (see Lemma \ref{TransFunOfTiltings}).

First note that $j([X_aU(\nu)]) = \widehat{X}_aU_{h}(1) = U_{f_{\la}}(1)$. Therefore, we may conclude that $U(\la)$ is a direct summand of $X_aU(\nu)$ by Lemma \ref{ExpressionOfUf}. We shall prove that $U(\la)=X_a U(\nu)$ by proving that $X_a U(\nu)$ is indecomposable.

Suppose $X_a U(\nu)$ is decomposable. Let $ X_a U(\nu)=T_1\oplus T_2$ with $T_1=U(\la)$. It follows from Lemma \ref{LemmaForProcedure}  that
  \begin{align*}
j([Y_aX_aU(\nu)]) =  \widehat{Y}_a\widehat{X}_aU_h(1)= \left\{ \begin{array}{ll}  U_h(1), \text{ if } \sharp \la = \sharp \nu,\\
2U_h(1), \text{ if } \sharp \la -1 = \sharp \nu.
 \end{array} \right.
   \end{align*}
 Since $\widehat{X}_a,\widehat{Y}_a$ are bi-adjoint to each other, as in the proof of \cite[Theorem 4.37]{Br1}, we have $(Y_aT_i:U(\nu))\neq 0$ for $i=1,2$. This means that $j([Y_aX_aU(\nu)]) =2U_h(1) $. Therefore,
\begin{align*}
Y_aX_a U(\nu)=U(\nu)\oplus U(\nu),
\end{align*}
and so  $Y_aT_1=Y_aT_2=U(\nu)$. We obtain $[Y_aU(\la):L(\nu)]=1$. We will show that $[Y_a U(\la):L(\nu)]\ge 2$ and so get a contradiction.

By Lemma \ref{LemmaForProcedure} again, there is  $\mu=\la-\alpha\in\La^+_{k,\zeta}$ with $\alpha\in\Phi^+(\mf u)$, $(\la,\ov{\alpha})=0$  such that $\widehat{X}_aK_h(1) = K_f(1)+K_{f_{\mu}}(1)$. By Corollary \ref{cor:tilting} we have
\begin{align*}
\left(U(\la):K(\mu)\right)=[K(-w_0^+\mu):L(-w_0^+\la)]=[K(-w_0^+\la+w_0^+\alpha):L(-w_0^+\la)].\end{align*} Note that \begin{align*} (-w_0^+\la+w_0^+\alpha ,\ov{w_0^+\alpha})=-(w_0^+\la,w_0^+\ov{\alpha})=-(\la,\ov{\alpha})=0. \end{align*} Consequently, by Theorem \ref{thm:fund:fact} we have $[K(-w_0^+\la+w_0^+\alpha):L(-w_0^+\la)]\ge 1$ and hence $\left(U(\la):K(\mu)\right)\ge 1$.

Since $\left(U(\la):K(\la)\right)= 1$ and $[K(\la):L(\mu)]\ge 1$ by Theorem \ref{thm:fund:fact}, we conclude that
\begin{align}\label{aux101}
\left[U(\la):L(\mu)\right]\ge 2.
\end{align}

Furthermore, since $X_aK(\nu)$ has a filtration with $K(\mu)$ on the top, by the adjunction between $\widehat{X}_a, \widehat{Y}_a$ again we have
\begin{align*}
\text{Hom}_\G\left(K(\nu),\widehat{Y}_a L(\mu)\right)= \text{Hom}_\G\left(\widehat{X}_a K(\nu), L(\mu)\right)\not=0,
\end{align*}
which implies that $[Y_aL(\mu):L(\nu)]\ge 1$. Finally, combining this with \eqref{aux101} gives $[Y_a U(\la):L(\nu)]\ge 2$.
\end{proof}

We are now ready to prove Theorem \ref{PKLConjecture}.

\begin{proof}[Proof of Theorem \ref{PKLConjecture}]
 By Theorem \ref{ConstructionOfTiltings} and Corollary \ref{cor:tilting} we have the multiplicity formula $u_{\mu,\la}(1) = (U(\la):K(\mu))$ and $u_{-w_0^+\la,-w_0^+\mu}(1) = (K(\la):L(\mu))$. Namely, we have the character formulas
 \begin{align*}
 &\text{ch}U(\la) = \sum_{\mu\preceq \la} u_{\mu,\la}(1)\text{ch}K(\mu),\\
 &\text{ch}K(\la) = \sum_{\mu\preceq \la} u_{-\omega_0^+ \la, -\omega_0^+ \mu}(1) \text{ch}L(\mu).
 \end{align*}
 Let $\textsf{1}_{k|n-k}: = \sum_{1\leq i\leq k}\delta_i -\sum_{k+1\leq i\leq n}\delta_i$. From \cite[Corollary 3.14 and (4.17)]{Br1}), we have that the following transition matrix
 \begin{align*}
 \left(u_{-\omega_0^+ \la, -\omega_0^+ \mu}(1)\right)_{\la,\mu \in \La^+_{k,\zeta}}
 \end{align*}
 has inverse matrix
 \begin{align*}
 \left(\ell_{ \mu +(n+1) \textsf{1}_{k|n-k},  \la+(n+1) \textsf{1}_{k|n-k}}(1)\right)_{\la,\mu \in \La^+_{k,\zeta}} =\left(\ell_{ \mu,  \la}(1)\right)_{\la,\mu \in \La^+_{k,\zeta}}.
 \end{align*}
 The completes the proof.
\end{proof}

\section{Kac-Wakimoto and Sergeev-Pragacz type character formulas}\label{sec:KW:formula}

In this section we apply Theorem \ref{PKLConjecture} to obtain closed character formula for analogues of Kostant and polynomial modules of $\mf{q}(n)$. We first recall the notation of $\h'_{m|n}$, $\delta'_a$  ($a\in I(m|n)$), and $\Phi'^+$ from Section \ref{SectionQueerLieSuperalgebra}. Furthermore, given a partition $\mu=(\mu_1,\mu_2,\ldots)$, we let $\mu^t$ denote its conjugate partition. Finally, recall that a partition $\mu$ is called a {$(k|n-k)$-hook partition} if $\mu_{k+1}\le n-k$.

Let $0\le k\le n$ and let $\la\in\La_{k,\zeta}$. Define $\rho=\sum_{i=1}^k (k-i+1-\frac{n+1}{2})\delta_i+\sum_{j=k+1}^n (k-j+\frac{n+1}{2})\delta_j$. Define $\la'=\sum_{i=1}^n\la'_i\delta_i$ by $$\la':=\sum_{i=1}^k(\la_i-\zeta-k+i-1+\frac{n+1}{2})\delta_i+\sum_{j=k+1}^n(\la_j+\zeta+j-k-\frac{n+1}{2})\delta_j.$$
Identifying $\delta_i$ with $\delta'_{-k-1+i}$ and $\delta_j$ with $\delta'_{j-k}$, for $1\le i\le k$ and $k+1\le j\le n$, we may regard $\la'$ and $\rho$ as elements in $\h'^*_{k|n-k}$ and thus as weights for $\gl(k|n-k)$. This gives a bijection between the set $\La_{k,\zeta}$ and the set of integral weights for $\gl(k|n-k)$. In this section we shall freely use this identification and thus identify $\h_\even^*$ with $\h_{k|n-k}'^*$.

Recall that the Borel subalgebras of a general linear Lie superalgebra $\gl(k|n-k)$ are in general not conjugate under its Weyl group $\mf S_{k|n-k}=\mf S_k\times \mf S_{n-k}$. However, it is well-known \cite{LSS} that any two non-conjugate Borel subalgebras with identical even subalgebra can be transformed to each other by a sequence of odd reflections. For a Borel subalgebra ${\mf b'}$ of $\gl(k|n-k)$ let us denote the set of positive and simple roots of ${\mf b'}$ by $\Phi'^+_{{\mf b'}}$ and $\Pi'_{{\mf b'}}$, respectively. Recall that the set of positive roots of the standard Borel subalgebra is denoted by $\Phi'^+$.

Let us denote the highest weight irreducible $\gl(k|n-k)$-module of highest weight $\nu$ with respect to the Borel subalgebra ${\mf b'}$ by $L'_{{\mf b'}}(\nu)$. Let $\rho_{{\mf b'}}$ denote the signed half sum of the positive roots in ${\mf b'}$. Above, the notation $\rho$ stands for the Weyl vector with respect to the standard Borel.

Recall the notion of a $\gl(k|n-k)$-Kostant module from \cite{BS}. In the language of \cite{SZ1} a finite-dimensional irreducible $\gl(k|n-k)$-module of highest weight (with respect to the standard Borel subalgebra) $\la$ is a Kostant module, if $\la$ is {\em totally connected}. By \cite{CHR} it follows that a finite-dimensional irreducible module $L'$ is a Kostant module if and only if there exists a weight $\nu$ and a Borel subalgebra ${\mf b'}$ with a distinguished subset $S\subseteq\Pi'_{{\mf b'}}$ consisting of mutually orthogonal roots such that (i) $L'\cong L'_{{\mf b'}}(\nu)$, (ii) $\sharp\nu=|S|$, and (iii) $S$ is orthogonal to $\nu+\rho_{\mf b'}$. Furthermore, the character for such a module is given by the so-called Kac-Wakimoto character formula which was conjectured in \cite{KW} and established (in the type $A$ case) in \cite{CHR}:

\begin{align}\label{KW:gl}
\text{ch} L'_{{\mf b'}}(\nu)=\frac{1}{\sharp\nu!}\frac{\prod_{\beta\in\Phi'^+_{{\mf b'},\bar{1}}}e^{\beta/2}+e^{-\beta/2}}{\prod_{\alpha\in\Phi'^+_{{\mf b'},\bar{0}}}e^{\alpha/2}-e^{-\alpha/2}} \sum_{w\in\mf S_{k|n-k}}(-1)^{\ell(w)}w\left(\frac{e^{\nu+\rho_{{\mf b'}}}}{\prod_{\gamma\in S}1+e^{-\gamma}}\right).
\end{align}

\begin{lem}\label{KW:KL:iden}
Let $\la\in\La^+_{k,\zeta}$ such that $L'(\la')$ is a $\gl(k|n-k)$-Kostant module. Suppose that $L'(\la')\cong L'_{{\mf b'}}(\la'_{{\mf b'}})$ such that $S\subseteq\Pi'_{\mf b'}$ is a distinguished subset consisting of mutually orthogonal roots with $\sharp\la'=|S|$ and orthogonal to $\la'_{\mf b'}+\rho_{\mf b'}$. Then we have the following identity in $\h_\even^*$:
\begin{align*}
\sum_{\mu\preceq\la}\ell_{\mu\la}(1)\sum_{w\in\mf S_{k|n-k}}(-1)^{\ell(w)}w(e^{\mu'+\rho}) = \frac{1}{\sharp\la!} \sum_{w\in\mf S_{k|n-k}}(-1)^{\ell(w)}w\left(\frac{e^{\la'_{{\mf b'}}+\rho_{{\mf b'}}}}{\prod_{\gamma\in S}1+e^{-\gamma}}\right).
\end{align*}
\end{lem}

\begin{proof}
Let $K'(\la')$ denote the Kac module of $\gl(k|n-k)$ of highest weight $\la'$ with respect to the standard Borel subalgebra. By \cite[Theorem 4.37]{Br1} we have
\begin{align*}
\text{ch}L'(\la')=\sum_{\mu\preceq\la}\ell_{\mu\la}(1)\text{ch}K'(\mu').
\end{align*}
Combining this with \eqref{KW:gl} we have the identity:
\begin{align*}
\sum_{\mu\preceq\la}&\ell_{\mu\la}(1) \frac{\prod_{\beta\in\Phi'^+_\odd}e^{\beta/2}+e^{-\beta/2}} {\prod_{\alpha\in\Phi'^+_\even}e^{\alpha/2}-e^{-\alpha/2}}\sum_{w\in\mf S_{k|n-k}}(-1)^{\ell(w)}w(e^{\mu'+\rho})=\\
&\frac{1}{\sharp\la'!}\frac{\prod_{\beta\in\Phi'^+_{{\mf b'},\bar{1}}}e^{\beta/2}+e^{-\beta/2}}{\prod_{\alpha\in\Phi'^+_{{\mf b'},\bar{0}}}e^{\alpha/2}-e^{-\alpha/2}} \sum_{w\in\mf S_{k|n-k}}(-1)^{\ell(w)}w\left(\frac{e^{\la_{{\mf b'}}'+\rho_{{\mf b'}}}}{\prod_{\gamma\in S}1+e^{-\gamma}}\right).
\end{align*}
Since the even subalgebra of ${\mf b'}$ and that of the standard Borel subalgebra coincide, we have
\begin{align*}
\frac{\prod_{\beta\in\Phi'^+_\odd}e^{\beta/2}+e^{-\beta/2}} {\prod_{\alpha\in\Phi'^+_\even}e^{\alpha/2}-e^{-\alpha/2}} =  \frac{\prod_{\beta\in\Phi'^+_{{\mf b'},\bar{1}}}e^{\beta/2}+e^{-\beta/2}}{\prod_{\alpha\in\Phi'^+_{{\mf b'},\bar{0}}}e^{\alpha/2}-e^{-\alpha/2}}.
\end{align*}
From this the lemma follows.
\end{proof}

Note that corresponding to the Borel subalgebra ${\mf b'}$ for $\gl(k|n-k)$ we have a Borel subalgebra of $\G=\mf{q}(n)$, which is obtained in a similar way as for $\gl(k|n-k)$ with the sequence of odd reflections replaced by the corresponding sequence of twisting functors \cite{Ch}.

For $\la\in\La^+_{k,\zeta}$ we call an irreducible $\mf q(n)$-module $L(\la)$ a {\em Kostant module}, if $L'(\la')$ is a Kostant module of $\gl(k|n-k)$. We can now prove the following Kac-Wakimoto type character formula for Kostant modules of $\mf q(n)$.

\begin{thm}\label{thm:KW:formula}
Let $\la\in\La^+_{k,\zeta}$ such that $L(\la)$ is a Kostant module. Let ${\mf b'}$ be the corresponding Borel subalgebra of $\gl(k|n-k)$ with a distinguished set $S\subseteq\Pi'_{{\mf b'}}$ consisting of mutually orthogonal roots and $\sharp\la'=\sharp\la=|S|$ and orthogonal to $\la'_{\mf b'}+\rho_{\mf b'}$. Let $\la_{{\mf b'}}=\la'_{{\mf b'}}+\rho_{{\mf b'}}+\zeta 1_{k|n-k}$. Then we have
\begin{align*}
\text{ch}L(\la)=\frac{2^{\lceil{n/2}\rceil}}{\sharp\la!} \prod_{\alpha\in\Phi^+}\frac{1+e^{-\alpha}}{1-e^{-\alpha}}\sum_{w\in\mf S_{k|n-k}}(-1)^{\ell(w)}w\left(\frac{e^{\la_{{\mf b'}}}}{\prod_{\gamma\in S}1+e^{-\gamma}}\right).
\end{align*}
\end{thm}

\begin{proof}
By Theorem \ref{PKLConjecture} we have $\text{ch}L(\la)=\sum_{\mu\preceq\la}\ell_{\mu\la}(1)\text{ch}K(\mu)$. Thus, we compute
\begin{align*}
\text{ch}L(\la)
&=2^{\lceil{n/2}\rceil}\sum_{\mu\preceq\la}\ell_{\mu\la}(1)\prod_{\beta\in\Phi(\mf u^+)}\frac{1+e^{-\beta}}{1-e^{-\beta}}\sum_{w\in \mf S_{k|n-k}}(-1)^{\ell(w)}w\left(e^\mu\right)\prod_{\beta\in\Phi^+(\mf l)}\frac{1+e^{-\beta}}{1-e^{-\beta}}\\
&=2^{\lceil{n/2}\rceil}\sum_{\mu\preceq\la}\ell_{\mu\la}(1)\prod_{\beta\in\Phi^+}\frac{1+e^{-\beta}}{1-e^{-\beta}}\sum_{w\in \mf S_{k|n-k}}(-1)^{\ell(w)}w\left(e^\mu\right)\\
&=2^{\lceil{n/2}\rceil}\sum_{\mu\preceq\la}\ell_{\mu\la}(1)\prod_{\beta\in\Phi^+}\frac{1+e^{-\beta}}{1-e^{-\beta}}\sum_{w\in \mf S_{k|n-k}}(-1)^{\ell(w)}w\left(e^{\mu'+\rho+\zeta 1_{k|n-k}}\right)\\
&=2^{\lceil{n/2}\rceil}\sum_{\mu\preceq\la}\ell_{\mu\la}(1)\prod_{\beta\in\Phi^+}\frac{1+e^{-\beta}}{1-e^{-\beta}}\sum_{w\in \mf S_{k|n-k}}(-1)^{\ell(w)}w\left(e^{\mu'+\rho}\right)e^{\zeta 1_{k|n-k}}\\
&= \frac{2^{\lceil{n/2}\rceil}}{\sharp\la!} \prod_{\beta\in\Phi^+}\frac{1+e^{-\beta}}{1-e^{-\beta}} \sum_{w\in\mf S_{k|n-k}}(-1)^{\ell(w)}w\left(\frac{e^{\la'_{{\mf b'}}+\rho_{{\mf b'}}}}{\prod_{\gamma\in S}1+e^{-\gamma}}\right) e^{\zeta1_{k|n-k}},
\end{align*}
where in the last identity we have used Lemma \ref{KW:KL:iden}. The theorem now follows.
\end{proof}

\begin{example}
Consider $\mf q(4)$ and $\la=(\zeta+2)\delta_1+(\zeta+1)\delta_2+(-\zeta-1)\delta_3+(-\zeta-2)\delta_4$ so that $k=2$ and $\sharp\la=2$. Furthermore, $\Phi^+=\{\delta_i-\delta_j|1\le i<j\le 4\}$ and the integral Weyl group here is $\mf S_2\times\mf S_2$, consisting of permutations on the letters $\{1,2\}$ and $\{3,4\}$. Then $\la_{{\mf b'}}=(\zeta+2)\delta_1+(\zeta+2)\delta_2+(-\zeta-2)\delta_3+(-\zeta-2)\delta_4$ and $S=\{\delta_1-\delta_3,\delta_2-\delta_4\}$. We have
\begin{align*}
\text{ch}L(\la)=2 \prod_{1\le i<j\le 4}\frac{1+e^{-\delta_i+\delta_j}}{1-e^{-\delta_i+\delta_j}}\sum_{w\in\mf S_2\times\mf S_{2}}(-1)^{\ell(w)}w\left(\frac{e^{(\zeta+2) 1_{2|2}}}{(1+e^{-\delta_1+\delta_3})(1+e^{-\delta_2+\delta_4})}\right).
\end{align*}
\end{example}

\begin{rem}
Theorem \ref{thm:KW:formula} suggests that the Kostant modules for $\mf{q}(n)$ have BGG type resolutions in terms of the parabolic Verma modules $K(\mu)$ in analogy to the resolution of $\gl(k|n-k)$-Kostant modules by Kac modules \cite{CKL, BS}.
\end{rem}

We recall that every irreducible polynomial module of $\gl(k|n-k)$, i.e., every irreducible submodule of a tensor power of the standard module $\C^{k|n-k}$, is a Kostant module. For such modules, recall that one has another closed classical character formula, called the Sergeev-Pragacz formula (see, e.g., \cite[Page 60]{Mac} or \cite[\S12.2]{Mu}). Below, we shall derive an analogue of this formula for $\mf q(n)$-Kostant modules that correspond to polynomial modules for the general linear Lie superalgebra.

It is well-known that the isomorphism classes of irreducible polynomial modules of the Lie superalgebra $\gl(k|n-k)$ are in bijection with the so-called $(k|n-k)$-hook partitions. To be more precise, let $\nu=\sum_{i=1}^n\nu_i\delta'_i\in\h_{k|n-k}'^*$. A necessary and sufficient condition for $\nu$ to the highest weight (with respect to the standard Borel subalgebra) of an irreducible polynomial representation is that $\nu^{-}=(\nu_1,\ldots,\nu_k)$ and $\nu^+=(\nu_{k+1},\ldots,\nu_{n})$ are both partitions, and in addition $(\nu^{-},(\nu^{+})^t)$ is a $(k|n-k)$-hook partition.

Let $L'(\nu)$ be a polynomial module of $\gl(k|n-k)$. Then we can visualize the corresponding hook partition diagrammatically as follows:
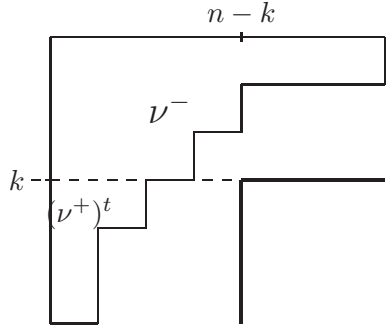
\begin{center}
\hskip 0cm \setlength{\unitlength}{0.25in}
\begin{picture}(7.5,6.5)
\put(0,0){\line(1,0){1}} \put(1,0){\line(0,1){2}}
\put(1,2){\line(1,0){1}} \put(2,2){\line(0,1){1}}
\put(2,3){\line(1,0){1}} \put(3,3){\line(0,1){1}}
\put(3,4){\line(1,0){1}} \put(4,4){\line(0,1){1}}
\put(4,5){\line(1,0){3}} \put(7,5){\line(0,1){1}}
\multiput(-.4,3)(0.4,0){13}{\line(1,0){0.2}}
\put(2.5,4.5){\makebox(0,0)[c]{\Large$\nu^-$}}
\put(0.6,2.3){\makebox(0,0)[c]{$(\nu^+)^t$}}
\put(7,6){\line(-1,0){7}} \put(0,6){\line(0,-1){6}}
\put(-.1,3){\line(1,0){0.2}} \put(-.7,3){\makebox(0,0)[c]{$k$}}
\put(4,6.1){\line(0,-1){0.2}} \put(4,6.5){\makebox(0,0)[c]{$n-k$}}
\put(4,3){\linethickness{1pt}\line(0,-1){3}}
\put(4,3){\linethickness{1pt}\line(1,0){3}}
\end{picture}
\end{center}
We can associate to the corresponding hook partition $\nu$ three partitions $M_\nu$, $r_\nu$, and $b_\nu=\nu^+$ as follows:
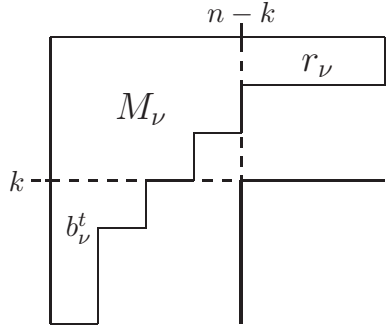
\begin{center}
\hskip 0cm \setlength{\unitlength}{0.25in}
\begin{picture}(7.5,6.5)
\put(0,0){\line(1,0){1}} \put(1,0){\line(0,1){2}}
\put(1,2){\line(1,0){1}} \put(2,2){\line(0,1){1}}
\put(2,3){\line(1,0){1}} \put(3,3){\line(0,1){1}}
\put(3,4){\line(1,0){1}} \put(4,4){\line(0,1){1}}
\put(4,5){\line(1,0){3}} \put(7,5){\line(0,1){1}}
\put(7,6){\line(-1,0){7}} \put(0,6){\line(0,-1){6}}
\multiput(-.4,3)(0.4,0){13}{\line(1,0){0.2}}
\multiput(4,1.5)(0,0.4){13}{\line(0,-1){0.2}}
\put(1.9,4.5){\makebox(0,0)[c]{\Large$M_\nu$}}
\put(5.6,5.4){\makebox(0,0)[c]{\Large$r_\nu$}}
\put(0.6,2){\makebox(0,0)[c]{$b_\nu^t$}}
\put(-.1,3){\line(1,0){0.2}} \put(-.7,3){\makebox(0,0)[c]{$k$}}
\put(4,6.1){\line(0,-1){0.2}} \put(4,6.5){\makebox(0,0)[c]{$n-k$}}
\put(4,3){\linethickness{1pt}\line(0,-1){3}}
\put(4,3){\linethickness{1pt}\line(1,0){3}}
\end{picture}
\end{center}

Let $x_i=e^{\delta_i}$, $i=1,\ldots,k$ and $y_j=e^{\delta_{k+j}}$, $j=1,\ldots, n-k$. We have the following Sergeev-Pragacz character formula for $L'(\nu)$:
\begin{align}\label{sergeev:pragacz}
\text{ch}L'(\nu)=\sum_{w\in\mf S_{k|n-k}} w\left( \frac{g_\nu(x,y)x^{r_\nu} y^{b_\nu}\prod_{i=1}^k x_i^{k-i}\prod_{j=1}^{n-k}y_j^{n-k-j}}{\Delta(x)\Delta(y)}\right),
\end{align}
where $g_\nu(x,y)=\prod_{(i,j)\in M_\nu}(x_i+y_j)$, $\Delta(x)=\prod_{i<j}(x_i-x_j)$, and $\Delta(y)=\prod_{p<q}(y_p-y_q)$. Here $x^{r_\nu}:=\prod_{i=1}^k x_i^{(r_\nu)_i}$ and $y^{b_\nu}:=\prod_{j=1}^{n-k} y_j^{(b_\nu)_j}$. (Also here we have used the identification between $\delta_i$s and $\delta'_j$ as explained above)

Let $C_\nu$ be the complement of $M_\nu$ in the $k\times (n-k)$ box, i.e., the Young diagram $\underbrace{(n-k,n-k,\ldots,n-k)}_k$.

\begin{lem}\label{SP:KL:iden}
Let $\la\in\La^+_{k,\zeta}$ such that $\la'$ is the highest weight of an irreducible polynomial module for $\gl(k|n-k)$. Then we have the following identity in $\h_\even^*$:
\begin{align*}
\sum_{w\in\mf S_{k|n-k}}(-1)^{\ell(w)}&w\left( \frac{x^{\la'^-} y^{{\la'^+}}\prod_{i=1}^k x_i^{k-i}\prod_{j=1}^{n-k}y_j^{n-k-j}}{\prod_{(i,j)\in C_{\la'}}1+x_i^{-1}y_j}\right) =\\
&\sum_{\mu\preceq\la}\ell_{\mu\la}(1)\sum_{w\in\mf S_{k|n-k}}(-1)^{\ell(w)}w\left(x^{\mu'^-} y^{{\mu'^+}}\prod_{i=1}^k x_i^{k-i}\prod_{j=1}^{n-k}y_j^{n-k-j}\right).
\end{align*}
\end{lem}

\begin{proof}
To simplify notation let us write $x^{\rho_x}:=\prod_{i=1}^k x_i^{k-i}$ and $y^{\rho_y}:=\prod_{j=1}^{n-k}y_j^{n-k-j}$. For an integer $l$ we write $x^l:=\prod_{i=1}^kx_i^l$ and similarly for $y^l$.

We have by \eqref{sergeev:pragacz}
\begin{align*}
\text{ch}L'(\la')=&\sum_{w\in \mf S_{k|n-k}} w\left( \frac{\prod_{(i,j)\in M_{\la'}}(x_i+y_j)x^{r_{\la'}}y^{b_{\la'}}x^{\rho_x}y^{\rho_y}}{\Delta(x)\Delta(y)} \right)\\
=&\sum_{w\in \mf S_{k|n-k}} (-1)^{\ell(w)}\frac{1}{\Delta(x)\Delta(y)}w\left( \frac{\prod_{i,j}(x_i+y_j)x^{r_{\la'}}y^{b_{\la'}}x^{\rho_x}y^{\rho_y}}{\prod_{(i,j)\in C_{\la'}}(x_i+y_j)} \right)\\
=&\sum_{w\in \mf S_{k|n-k}} (-1)^{\ell(w)}\frac{\prod_{i,j}(x_i+y_j)}{\Delta(x)\Delta(y)}w\left( \frac{x^{r_{\la'}}y^{b_{\la'}}x^{\rho_x}y^{\rho_y}}{\prod_{(i,j)\in C_{\la'}}(x_i+y_j)} \right)\\
=&\sum_{w\in \mf S_{k|n-k}} (-1)^{\ell(w)}\frac{\prod_{i,j}(x_i+y_j)}{\Delta(x)\Delta(y)}w\left( \frac{x^{r_{\la'}}y^{b_{\la'}}x^{\rho_x}y^{\rho_y}}{x^{C_{\la'}}\prod_{(i,j)\in C_{\la'}}(1+x^{-1}_iy_j)} \right) \\
=&\sum_{w\in \mf S_{k|n-k}} (-1)^{\ell(w)}\frac{\prod_{i,j}(x_i+y_j)}{\Delta(x)\Delta(y)}x^{-n+k}w\left( \frac{x^{{\la'^-}}y^{b_{\la'}}x^{\rho_x}y^{\rho_y}}{\prod_{(i,j)\in C_{\la'}}(1+x^{-1}_iy_j)} \right).
\end{align*}

Also by \cite[Theorem 4.37]{Br1} we have
\begin{align*}
\text{ch}L'(\la')=\sum_{ \mu\preceq\la}\ell_{\mu\la}(1)\prod_{i,j}(x_i+y_j)\frac{x^{-n+k}}{\Delta(x)\Delta(y)}\sum_{w\in \mf S_{k|n-k}}(-1)^{\ell(w)}\left( x^{\mu'^-}y^{\mu'^+}x^{\rho_x}y^{\rho_y} \right).
\end{align*}
Comparing these two expressions the lemma follows.
\end{proof}

\begin{thm} Let $\la\in\La^+_{k,\zeta}$ such that $\la'$ is the highest weight of an irreducible polynomial module for $\gl(k|n-k)$. Then we have
\begin{align*}
\text{ch}L(\la) = &\frac{2^{\lceil n/2\rceil}\prod_{i<j}(x_i+x_j)\prod_{p<q}(y_p+y_q)}{\prod_{i,j}(x_i-y_j)}
x^{\zeta+\frac{n+1}{2}-k}
y^{-\zeta-\frac{n-1}{2}+k}\\
&\quad\times\sum_{w\in\mf S_{k|n-k}} w\left(\frac{g_{\la'}(x,y)x^{r_{\la'}} y^{b_{\la'}}\prod_{i=1}^k x_i^{k-i}\prod_{j=1}^{n-k}y_j^{n-k-j}}{\Delta(x)\Delta(y)} \right).
\end{align*}
\end{thm}

\begin{proof}
We define
\begin{align*}
\kappa:=\sum_{i=1}^k(\zeta-\frac{n-1}{2})\delta_i +\sum_{j=k+1}^n(k-\frac{n-1}{2}-\zeta)\delta_j.
\end{align*}
so that we have $\la=\la'+\rho_x+\rho_y+\kappa$.
By Theorem \ref{PKLConjecture} and Lemma \ref{SP:KL:iden} we have the following expression for $\text{ch}L(\la)$:
\begin{align*}
&=2^{\lceil n/2\rceil}\sum_{\mu\preceq\la}\ell_{\mu\la}(1)\prod_{i,j}\frac{x_i+y_j}{x_i-y_j} \frac{\prod_{i<j,p<q}(x_i+x_j)(y_p+y_q)}{\Delta(x)\Delta(y)}\\
&\qquad\qquad\qquad\qquad \qquad\qquad\times e^{\kappa}\sum_{w}(-1)^{\ell(w)}w\left(x^{\mu'^-}y^{\mu'^+}x^{\rho_x}y^{\rho_y}\right)\\
&=2^{\lceil n/2\rceil}e^{\kappa}\prod_{i,j}\frac{x_i+y_j}{x_i-y_j} \frac{\prod_{i<j,p<q}(x_i+x_j)(y_p+y_q)}{\Delta(x)\Delta(y)}\\
&\qquad\qquad\qquad\qquad \qquad\qquad\times \sum_{w}(-1)^{\ell(w)}w\left( \frac{x^{\la'^-} y^{{\la'^+}}x^{\rho_x}y^{\rho_y}}{\prod_{(i,j)\in C_{\la'}}1+x_i^{-1}y_j}\right)\\
&=2^{\lceil n/2\rceil}e^{\kappa} \frac{\prod_{i<j,p<q}(x_i+x_j)(y_p+y_q)}{\prod_{i,j}(x_i-y_j)}\sum_{w}w\left( \frac{\prod_{i,j}(x_i+y_j)x^{\la'^-} y^{{\la'^+}}x^{\rho_x}y^{\rho_y}x^{C_{\la'}}}{\Delta(x)\Delta(y)\prod_{(i,j)\in C_{\la'}}x_i+y_j}\right)\\
&=2^{\lceil n/2\rceil}x^{n-k}e^{\kappa} \frac{\prod_{i<j,p<q}(x_i+x_j)(y_p+y_q)}{\prod_{i,j}(x_i-y_j)}\\
&\qquad\qquad\qquad\qquad \qquad\qquad\times\sum_{w\in\mf S_{k|n-k}}
w\left( \frac{\prod_{(i,j)\in M_{\la'}}(x_i+y_j)x^{r_{\la'}} y^{{b_{\la'}}}x^{\rho_x}y^{\rho_y}}{\Delta(x)\Delta(y)}\right).
\end{align*}
Recalling the definitions of $\kappa$ and $g_{\la'}(x,y)$ gives the theorem.
\end{proof}

\begin{rem}
Consider the full subcategory of $\mc O_{n,\hf+\Z}$ consisting of objects with composition factors isomorphic to $L(\la)$ with $\la=\sum_{i=1}^n\la_i\delta_i\in\h_\even^*$ of the form $\la_i\in\hf\Z$ and $\la_{k+1}>\la_{k+2}>\cdots>\la_n>0>\la_1>\la_2>\cdots>\la_k$. According to \cite[Proposition 4.1 and Corollary 4.2]{CKW} the canonical basis on the corresponding subspace of the Fock space of type $C$ can be identified naturally with  canonical basis of type $A$. Now, a verbatim repetition of the arguments given above can be used to obtain an irreducible character formula for $L(\la)$ in analogy to Theorem \ref{PKLConjecture}. Here, we use $\hf$ for $\zeta$ in the expression \eqref{aux:fns} to define the corresponding Kazhdan-Lusztig polyomials $\ell_{\la\mu}(q)$. This establishes a parabolic version of a special case of the conjecture on the irreducible characters for the half-integer weights in \cite{CKW}.  Also, the formula for Kostant modules and analogues of polynomial modules in this section have analogues in this setting as well. We leave the details to the reader.

We expect that the characters of $L(\la)$ in the case when $\la$ satisfies the more general condition of $\la_j>0>\la_i$, for $i=1,\ldots,k$ and $j=k+1,\ldots,n$, and either $\la_l\in\hf\Z$ or $\la_l\in\Z$, for all $l$, are determind by canonical basis of type $A$ quantum groups. This is predicted by \cite{CKW} and one should be able to establish this following the approach in \cite{BLW}.
\end{rem}

\end{document}